\documentclass[11pt,a4paper]{article}
\usepackage[margin=2.6cm]{geometry}
\usepackage{amsmath,amsthm,amssymb,amsfonts,mathtools}
\usepackage{booktabs,array}
\usepackage{enumitem}
\usepackage{tikz}
\usepackage{microtype}
\usepackage[hidelinks]{hyperref}
\theoremstyle{plain}
\newtheorem{theorem}{Theorem}[section]
\newtheorem{lemma}[theorem]{Lemma}
\newtheorem{corollary}[theorem]{Corollary}
\newtheorem{proposition}[theorem]{Proposition}
\newtheorem{fact}[theorem]{Fact}
\theoremstyle{definition}
\newtheorem{example}[theorem]{Example}

\DeclareMathOperator{\per}{per}
\DeclareMathOperator*{\Ex}{\mathbb{E}}
\newcommand{\LL}{\mathsf L}
\newcommand{\RR}{\mathsf R}
\newcommand{\JJ}{\mathbb J}

\title{\bfseries Rainbow maps and the number of graceful labelings of trees}
\author{Shisheng Li}
\date{}
\begin{document}\maketitle

\begin{abstract}\noindent
Let $\tau(N)$ be the number of graceful labelings of trees with $N$ edges, that is, the number of spanning trees of the complete graph on $\{0,1,\dots,N\}$ whose $N$ edges have distinct lengths $|i-j|$.
The answer to exercise 7.2.2.3--106 of \emph{The Art of Computer Programming} states that, if an empirical observation of Anick is true, then $\tau(n)=\exp(n\ln n-n\ln(e/b)+O(\log n))$ with $b\approx0.636$, and that rigorous proofs are unknown.
We prove that
\[
  \Bigl|\log\tau(N)-\Bigl(N\log N-N\log\frac eb\Bigr)\Bigr|\le152\,\bigl(1+\log(N+1)\bigr)^4\qquad(N\ge3),\qquad b=\frac2\pi .
\]
A \emph{rainbow map} of order $N$ sends each interior vertex $k\in\{1,\dots,N-1\}$ to a vertex $f(k)\in\{0,\dots,N\}$ so that the lengths $|f(k)-k|$ are $1,\dots,N-1$ in some order. The acyclic rainbow maps are exactly the graceful labelings, and the number of all rainbow maps is the Whitty number $W(N)$ (OEIS A134983).
We show that $\tau(N)\le W(N)\le\tau(N)\,e^{151(1+\log(N+1))^4}$ by a combinatorial argument: For fixed directions, rainbow maps are rook placements on a Ferrers board; a swap injection shows that a uniformly random rainbow map has $O(\log^2N)$ cycles on average; and an inequality on the periods of a word shows that the $c$ cycles of a rainbow map can be removed by changing $O(c\log N)$ edges.
We also show that $W(N)$ is the total weight of a family of weighted Motzkin paths. This yields exact enumerations of types, an alternating sum for $W(N)$, and a representation of $W(N)$ as a moment of a continuous dual Hahn weight, from which we obtain $|\log W(N)-(N\log\frac{2N}\pi-N)|\le5(1+\log(N+1))$.
The main theorem has been formalized in Lean~4, with two cited identities for the Gamma function taken as axioms; the combinatorial half uses no axioms beyond those of Lean.
\end{abstract}

%======================================================================
\section{Introduction}
\subsection{The problem}
A \emph{graceful labeling} of a tree with $N$ edges is a bijection from its vertices to $\{0,1,\dots,N\}$ such that the $N$ edge lengths $|i-j|$ ($i$ and $j$ being the labels of the two ends) are distinct. Using the labels as vertex names, a graceful labeling is a spanning tree of the complete graph on $\{0,\dots,N\}$ whose edges have distinct lengths; we define $\tau(N)$ to be the number of such spanning trees. (Thus a labeling and its complement $i\mapsto N-i$ are counted separately unless they give the same tree.) With this definition $\tau(N)=\mathrm{A033472}(N+1)$ \cite{oeis}.
Exercise 7.2.2.3--106 of \emph{The Art of Computer Programming} \cite{taocp7}, rated \textsf{[HM46]}, asks for the asymptotics of $\tau(n)$. The answer in the first printing reads:
\begin{quote}\small
\textbf{106.} Empirical investigations by D.\ Anick suggest that $\tau(n)/\tau(n-1)$ grows
approximately as $a+bn+(-1)^n c/n$ for some constants $a$, $b$, $c$. If that is true,
$\tau(n)=\exp(n\ln n-n\ln(e/b)+O(\log n))$. The exact values for $n<30$ suggest further that
$a\approx0.19$, $b\approx0.636$, and $c\approx0.42$. But rigorous proofs are unknown.
(This function $\tau(n)$ was introduced by A.\ Kotzig, who computed it by hand for
$n\le6$ in 1984.)
\end{quote}
Anick's study appeared in \cite{anick}.

\subsection{Results}\label{ssec:results}
Throughout, $m:=N-1$, $\Lambda:=1+\log(N+1)$, and $\log$ is the natural logarithm. (The constants $a$, $b$, $c$ of the quoted answer appear only here and in Section~\ref{sec:concl}; elsewhere these letters have local meanings, and $c$ usually denotes a number of cycles.)
Let $W(N)$ be the number of rainbow maps of order $N$ (Section~\ref{sec:setup}); it is the permanent of an explicit matrix with entries $0$, $1$, $2$.
\begin{theorem}[Comparison]\label{thm:compare}
For $N\ge3$, $\quad\tau(N)\le W(N)\le\tau(N)\,\exp\bigl(151\,\Lambda^4\bigr)$.
\end{theorem}
\begin{theorem}[Size of the Whitty number $W(N)$]\label{thm:perm}
For $N\ge3$,
\[
  \frac17\Bigl(\frac2\pi\Bigr)^{m}m!\ \le\ W(N)\ \le\ 45\Bigl(\frac2\pi\Bigr)^{m}(m+3)!,\qquad\text{hence}\qquad
  \Bigl|\log W(N)-\Bigl(N\log\frac{2N}\pi-N\Bigr)\Bigr|\ \le\ 5\Lambda .
\]
\end{theorem}
\begin{theorem}[Main theorem]\label{thm:main}
For every $N\ge3$,
\[
  \Bigl|\log\tau(N)-\Bigl(N\log\frac{2N}\pi-N\Bigr)\Bigr|\ \le\ 152\,\Lambda^4 ,
\]
and $N\log\frac{2N}\pi-N=N\log N-N\log\frac eb$ with $b=\frac2\pi$. In particular $\tau(N)=N!\,(2/\pi)^Ne^{O(\log^4N)}$.
\end{theorem}
\begin{proof}[Proof of Theorem~\ref{thm:main} from Theorems~\ref{thm:compare} and~\ref{thm:perm}]
Theorem~\ref{thm:compare} gives $-151\Lambda^4\le\log\tau(N)-\log W(N)\le0$, and Theorem~\ref{thm:perm} gives $|\log W(N)-(N\log\frac{2N}\pi-N)|\le5\Lambda$. Add them and use $5\Lambda\le\Lambda^4$, which holds because $\Lambda\ge2.38$ for $N\ge3$. The last assertion follows from \eqref{eq:fact} below.
\end{proof}
\noindent Section~\ref{sec:struct} also gives a Motzkin path formula for $W(N)$ (Theorem~\ref{thm:motzkin}) and several exact enumerations derived from it.

\subsection{Approach}
A graceful labeling is a spanning tree subject to two constraints: It has no cycle, and its edge lengths are distinct. Every such tree contains the edge $\{0,N\}$, the only edge of length $N$. Delete this edge and orient every other edge towards whichever of $0$ and $N$ lies in its component; then every vertex other than $0$ and $N$ has exactly one outgoing edge, and we obtain a rainbow map. If we drop the condition ``no cycle'', the remaining object is of product type: Each vertex chooses an outgoing edge, subject only to the lengths being distinct. The number of such choices is a permanent.
This relaxation is due to Whitty \cite{whitty}.
Theorem~\ref{thm:compare} shows that the relaxation loses almost nothing, and Theorems~\ref{thm:motzkin} and~\ref{thm:perm} compute the permanent.

\emph{Outline.} Section~\ref{sec:setup} defines rainbow maps, their types, and the Whitty numbers. Section~\ref{sec:struct} reads a type as a lattice path. Sections~\ref{sec:cycles}--\ref{sec:count} prove Theorem~\ref{thm:compare}: There are few cycles, cycles are easy to repair, and a counting argument finishes the proof. Section~\ref{sec:perm} determines the size of $W(N)$.

\emph{Formal verification.} The main theorem has been formalized in Lean~4 \cite{lean4} with Mathlib \cite{mathlib} (ancillary files, directory \texttt{anc/lean}). The formal statement uses only the definition of $\tau(N)$ as the number of spanning trees of the complete graph on $\{0,\dots,N\}$ with edge lengths $1,\dots,N$. The combinatorial half (Theorem~\ref{thm:compare}) is proved with no axioms beyond those of Lean; the analytic half uses Facts~\ref{fact:W} and~\ref{fact:P} as axioms, and Theorem~\ref{thm:fine} also uses Fact~\ref{fact:S}.

\emph{Notation.} $N$: number of edges, also the order of a rainbow map; $m=N-1$: number of interior vertices, also the number of available lengths; $\rho(k)=\min(k,N-k)$: the rank of $k$; $R=\lfloor N/2\rfloor$; $\delta_k$: the length used by $k$; $\sigma$: type; $\ell_k$: the upper bound of $\delta_k$ for a given type; $c$: number of cycles; $\mathcal H_R=\sum_{j\le R}1/j$; $D=\lceil\log_2m\rceil$ (for $N\ge2$); $\Lambda=1+\log(N+1)$; $[P]$ is $1$ if the statement $P$ is true and $0$ otherwise.

%======================================================================
\section{Rainbow maps}\label{sec:setup}
The vertices are $0,1,\dots,N$. We call $0$ and $N$ the \emph{ground} and $1,\dots,m$ the \emph{interior vertices}. The \emph{rank} of an interior vertex $k$ is $\rho(k):=\min(k,N-k)$, its distance to the nearer end. Each rank is shared by at most two vertices ($k$ and $N-k$), and there are $R:=\lfloor N/2\rfloor$ ranks.

\begin{quote}
\textbf{Definition.} A \emph{rainbow map} of order $N$ is a function $f:\{1,\dots,m\}\to\{0,\dots,N\}$ such that the numbers $|f(k)-k|$ are $1,\dots,m$, each exactly once.
\end{quote}
The name \emph{rainbow map} is due to D.~E.~Knuth (private communication). We call $f(k)$ the \emph{target} of $k$ and $\delta_k:=|f(k)-k|$ its \emph{length}. Drawing the arcs $k\to f(k)$ and the arc $0\to N$ gives a directed graph in which every interior vertex has exactly one outgoing arc.
The \emph{type} $\sigma\in\{\LL,\RR\}^m$ records the directions: $\sigma_k=\LL$ if and only if $f(k)<k$. Given the type, the length of $k$ is bounded by
\[
  \ell_k:=\begin{cases}k,&\sigma_k=\LL,\\ N-k,&\sigma_k=\RR,\end{cases}\qquad 1\le\delta_k\le\ell_k ,
\]
and $f$ is determined by $\sigma$ and $(\delta_1,\dots,\delta_m)$.

\textbf{Whitty numbers.} Let $W(N)$ be the number of rainbow maps of order $N$. If vertex $k$ receives length $d$, the number of possible directions is
\[
  M_{k,d}:=[d\le k]+[d\le N-k]\in\{0,1,2\},\qquad 1\le k,d\le m ,
\]
so that
\[
  W(N)=\sum_\delta\prod_{k=1}^{m}M_{k,\delta_k}=\per M ,
\]
the sum being over all permutations $\delta$ of $\{1,\dots,m\}$. For $N=6$,
\[
  M=\begin{pmatrix}2&1&1&1&1\\2&2&1&1&0\\2&2&2&0&0\\2&2&1&1&0\\2&1&1&1&1\end{pmatrix}
\]
and $W(6)=168$. Rows $k$ and $N-k$ are equal, and vertices of small rank can use more lengths. The sequence $W(N)$ begins $1,2,4,12,40,168,784,4272,\dots$ ($N=1,2,\dots$); it is OEIS A134983.
Whitty \cite[Theorem~7]{whitty} interprets this permanent as the number of rook placements on a Klein bottle chessboard and writes ``We do not know of a way of evaluating the permanent involved in closed form''. We call $W(N)$ the \emph{Whitty numbers}.

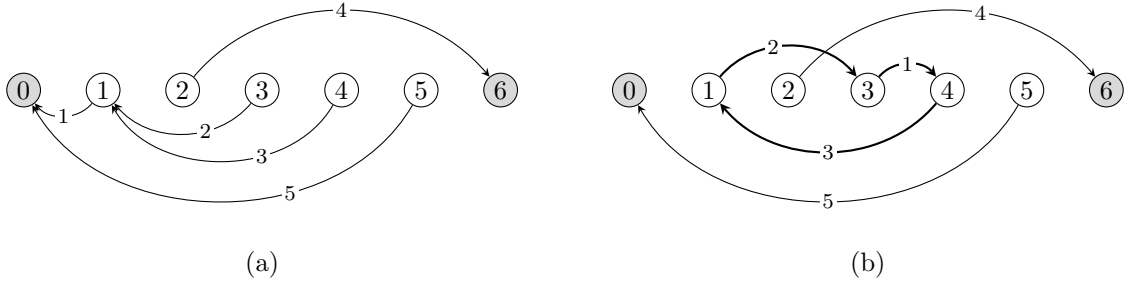
\begin{figure}[t]\centering
\begin{tikzpicture}[x=1.05cm,>=stealth]
 \begin{scope}
 \foreach \i in {1,...,5} \node[circle,draw,inner sep=1.6pt] (a\i) at (\i,0) {\small\i};
 \node[circle,draw,fill=black!15,inner sep=1.6pt] (a0) at (0,0) {\small0}; \node[circle,draw,fill=black!15,inner sep=1.6pt] (a6) at (6,0) {\small6};
 \draw[->] (a1) to[bend left=50] node[pos=.5,fill=white,inner sep=1pt,font=\scriptsize]{1} (a0);
 \draw[->] (a3) to[bend left=50] node[pos=.35,fill=white,inner sep=1pt,font=\scriptsize]{2} (a1);
 \draw[->] (a4) to[bend left=60] node[pos=.35,fill=white,inner sep=1pt,font=\scriptsize]{3} (a1);
 \draw[->] (a2) to[bend left=50] node[pos=.5,fill=white,inner sep=1pt,font=\scriptsize]{4} (a6);
 \draw[->] (a5) to[bend left=60] node[pos=.35,fill=white,inner sep=1pt,font=\scriptsize]{5} (a0);
 \node at (3,-2.3) {\small (a)};
 \end{scope}
 \begin{scope}[xshift=8cm]
 \foreach \i in {1,...,5} \node[circle,draw,inner sep=1.6pt] (b\i) at (\i,0) {\small\i};
 \node[circle,draw,fill=black!15,inner sep=1.6pt] (b0) at (0,0) {\small0}; \node[circle,draw,fill=black!15,inner sep=1.6pt] (b6) at (6,0) {\small6};
 \draw[->,thick] (b1) to[bend left=50] node[pos=.4,fill=white,inner sep=1pt,font=\scriptsize]{2} (b3);
 \draw[->,thick] (b3) to[bend left=50] node[pos=.5,fill=white,inner sep=1pt,font=\scriptsize]{1} (b4);
 \draw[->,thick] (b4) to[bend left=50] node[pos=.5,fill=white,inner sep=1pt,font=\scriptsize]{3} (b1);
 \draw[->] (b2) to[bend left=50] node[pos=.6,fill=white,inner sep=1pt,font=\scriptsize]{4} (b6);
 \draw[->] (b5) to[bend left=60] node[pos=.5,fill=white,inner sep=1pt,font=\scriptsize]{5} (b0);
 \node at (3,-2.3) {\small (b)};
 \end{scope}
\end{tikzpicture}
\caption{Two rainbow maps of order $6$. Arcs to the right are drawn above the line, arcs to the left below it; the numbers are the lengths, each of $1,\dots,5$ used once. The ground vertices $0$ and $6$ are shaded.
(a) Together with the edge $\{0,6\}$ this is a gracefully labeled tree. (b) A rainbow map with the cycle $1\to3\to4\to1$.}\label{fig:config}
\end{figure}

In the directed graph the two ground vertices are absorbing (apart from $0\to N$). Let $c(f)$ be the number of directed cycles. Every cycle has at least $3$ vertices, because $u\to v\to u$ would use the length $|u-v|$ twice.

\begin{lemma}[cf.\ {\cite[Theorem~6 and the remarks before Theorem~7]{whitty}}]\label{lem:inj}
The underlying undirected graph of $f$, together with the edge $\{0,N\}$, is a tree if and only if $c(f)=0$. Acyclic rainbow maps correspond bijectively to graceful labelings; hence $\tau(N)=\#\{f:c(f)=0\}\le W(N)$.
\end{lemma}
\begin{proof}
Given a graceful labeling, delete the edge $\{0,N\}$; the tree splits into two trees containing $0$ and $N$ respectively. Orient every edge towards the ground vertex of its tree. Every interior vertex has exactly one outgoing edge and the $m$ lengths are $1,\dots,m$, so this is an acyclic rainbow map.
Conversely, the $m$ edges of an acyclic rainbow map together with $\{0,N\}$ are $N$ edges on $N+1$ vertices. Following arcs from any interior vertex never closes a cycle, so it reaches the ground; the graph is connected, hence a tree, and its lengths are $1,\dots,N$. The two maps are inverse to each other. Finally, a directed cycle has at least $3$ vertices, so it is a cycle of the underlying graph, which is then not a tree.
\end{proof}
\noindent Thus $\tau(N)/W(N)=\Pr[c=0]$ for the uniform measure on all rainbow maps of order $N$. Reversing a cycle of a rainbow map gives another rainbow map with the same edges.

\begin{example}\label{ex:small}
$W(5)=40=\tau(5)$. For $N=6$, $W(6)=168$ and $\tau(6)=164$; the four extra rainbow maps each contain a triangle, for example the map of Figure~\ref{fig:config}(b). Table~\ref{tab:values} lists $\tau(N)$, $W(N)$ and their ratio for $N\le30$.
\end{example}
\begin{table}[t]\centering\footnotesize
\caption{$\tau(N)$ (from OEIS A033472; recomputed by exhaustive search for $N\le12$) and $W(N)$ (computed exactly by Theorem~\ref{thm:altsum}; for $N\le15$ in agreement with the terms listed in OEIS A134983).}\label{tab:values}
\begin{tabular}{@{}rrrr@{}}
\toprule
$N$ & $\tau(N)$ & $W(N)$ & $\tau(N)/W(N)$\\
\midrule
1 & 1 & 1 & 1.0000\\
2 & 2 & 2 & 1.0000\\
3 & 4 & 4 & 1.0000\\
4 & 12 & 12 & 1.0000\\
5 & 40 & 40 & 1.0000\\
6 & 164 & 168 & 0.9762\\
7 & 752 & 784 & 0.9592\\
8 & 4\,020 & 4\,272 & 0.9410\\
9 & 23\,576 & 25\,504 & 0.9244\\
10 & 155\,632 & 171\,168 & 0.9092\\
11 & 1\,112\,032 & 1\,242\,688 & 0.8949\\
12 & 8\,733\,628 & 9\,915\,072 & 0.8808\\
13 & 73\,547\,332 & 84\,714\,112 & 0.8682\\
14 & 670\,789\,524 & 783\,454\,848 & 0.8562\\
15 & 6\,502\,948\,232 & 7\,697\,178\,880 & 0.8448\\
16 & 67\,540\,932\,632 & 80\,965\,577\,472 & 0.8342\\
17 & 740\,949\,762\,580 & 898\,994\,002\,432 & 0.8242\\
18 & 8\,634\,364\,751\,264 & 10\,599\,312\,542\,208 & 0.8146\\
19 & 105\,722\,215\,202\,120 & 131\,228\,986\,811\,392 & 0.8056\\
20 & 1\,366\,258\,578\,159\,064 & 1\,714\,108\,476\,337\,152 & 0.7971\\
21 & 18\,468\,456\,090\,865\,364 & 23\,410\,457\,140\,701\,184 & 0.7889\\
22 & 262\,118\,487\,952\,306\,820 & 335\,566\,486\,059\,952\,128 & 0.7811\\
23 & 3\,877\,299\,383\,532\,914\,224 & 5\,011\,182\,724\,307\,685\,376 & 0.7737\\
24 & 59\,953\,554\,967\,155\,428\,756 & 78\,206\,210\,287\,187\,177\,472 & 0.7666\\
25 & 963\,196\,316\,865\,520\,244\,324 & 1\,267\,646\,452\,094\,243\,676\,160 & 0.7598\\
26 & 16\,118\,315\,603\,067\,394\,684\,460 & 21\,396\,344\,618\,478\,777\,704\,448 & 0.7533\\
27 & 279\,476\,688\,842\,701\,986\,223\,480 & 374\,097\,809\,126\,332\,231\,991\,296 & 0.7471\\
28 & 5\,032\,064\,364\,362\,833\,199\,584\,976 & 6\,790\,389\,880\,757\,556\,630\,700\,032 & 0.7411\\
29 & 93\,661\,896\,978\,248\,342\,171\,168\,192 & 127\,381\,472\,196\,794\,982\,167\,904\,256 & 0.7353\\
30 & 1\,805\,490\,553\,059\,186\,252\,316\,964\,240 & 2\,474\,264\,427\,361\,959\,296\,151\,748\,608 & 0.7297\\
\bottomrule
\end{tabular}

\end{table}

\subsection{Types and Ferrers boards}\label{ssec:ferrers}
For a given type $\sigma$, the set of lengths allowed for vertex $k$ is a \emph{prefix} $[1,\ell_k]$. So the rainbow maps of type $\sigma$ are the ways to assign the $m$ lengths to the $m$ rows without repetition, row $k$ taking a length in $[1,\ell_k]$; they are the rook placements on a Ferrers board.
Sort the rows as $\ell_{(1)}\le\dots\le\ell_{(m)}$. If rooks are placed row by row in this order, all earlier rooks lie in the range of the current row, so
\begin{equation}\label{eq:ferrers}
  \#\{f\text{ of type }\sigma\}=\prod_{i=1}^{m}\bigl(\ell_{(i)}-i+1\bigr)
\end{equation}
(which is $0$ if some factor is $\le0$). A type is \emph{feasible} if this number is positive. The uniform measure conditioned on a feasible type is the uniform measure on rook placements of its board.

\begin{figure}[t]\centering
\begin{tikzpicture}[scale=0.62]
 \foreach \r/\c in {0/5,1/4,2/3,3/4,4/5} { \foreach \x in {1,...,\c} \fill[black!12] (\x-1,-\r) rectangle (\x,-\r-1); }
 \draw (0,0) grid (5,-5);
 \foreach \x in {1,...,5} \node[font=\scriptsize] at (\x-0.5,0.4) {\x};
 \node[font=\scriptsize] at (2.5,1.1) {length $d$};
 \foreach \r/\v/\c/\dir in {0/1/5/\RR,1/2/4/\RR,2/3/3/\LL,3/4/4/\LL,4/5/5/\LL} \node[font=\scriptsize,anchor=east] at (-0.2,-\r-0.5) {vertex \v\ ($\dir$, $\ell=\c$)};
 \foreach \r/\x in {0/3,1/4,2/1,3/2,4/5} \fill (\x-0.5,-\r-0.5) circle (0.16);
\end{tikzpicture}
\caption{The board of type $\RR\RR\LL\LL\LL$ for $N=6$: Row $k$ is vertex $k$, and its shaded cells are the allowed lengths $1,\dots,\ell_k$.
The dots show the rainbow map $(\delta_1,\dots,\delta_5)=(3,4,1,2,5)$. Sorting the rows by $\ell$ (vertices $3,2,4,1,5$) gives a Ferrers diagram;
row by row the numbers of choices are $3$, $4-1$, $4-2$, $5-3$, $5-4$, so this board carries $3\cdot3\cdot2\cdot2\cdot1=36$ rainbow maps.}\label{fig:ferrers}
\end{figure}

%======================================================================
\section{The structure of rainbow maps}\label{sec:struct}
This section regroups formula \eqref{eq:ferrers} by rank. A type then becomes a lattice path, and $W(N)$ becomes the total weight of a family of weighted Motzkin paths (Theorem~\ref{thm:motzkin}). The paths give exact enumerations of types (Section~\ref{ssec:typecount}) and are the starting point of the computation of $W(N)$ in Section~\ref{sec:perm}. Only Theorem~\ref{thm:motzkin} is used later (in Section~\ref{sec:perm}); the rest of this section, on exact enumerations, on a uniformly random rainbow map of a fixed type, and on a few further statistics, is not needed for the main theorem.

\subsection{Pairs and lattice paths}\label{ssec:pairs}
Let $R':=\lfloor(N-1)/2\rfloor$. For $1\le r\le R'$ we pair the two vertices $r$ and $N-r$ of rank $r$. Their bounds are $r$ or $N-r$: Vertex $r$ pointing left, or vertex $N-r$ pointing right, has the \emph{small bound} $r$; vertex $r$ pointing right, or vertex $N-r$ pointing left, has the \emph{large bound} $N-r$. In other words, an arc with the small bound points to the nearer end, and an arc with the large bound points to the farther end.
Let $s_r\in\{0,1,2\}$ be the number of vertices of pair $r$ with the small bound ($s_r=1$ arises in two ways). If $N$ is even there is also the midpoint $N/2$, whose bound is $N/2$ in both directions. Let
\[
  S_0:=0,\qquad S_r:=s_1+\dots+s_r\quad(1\le r\le R'),\qquad x_r:=r-S_{r-1}\quad(1\le r\le R'+1),
\]
so that $x_1=1$ and $x_{r+1}=x_r+1-s_r$; the values $s_r=0$, $1$, $2$ are an up step, a level step and a down step.

\begin{theorem}[Motzkin paths]\label{thm:motzkin}
The number of rainbow maps of a given type is
\[
  \prod_{r=1}^{R'}w_{s_r}(x_r)\times\begin{cases}1,&N\text{ odd},\\ x_{R'+1},&N\text{ even},\end{cases}\qquad
  w_0(x)=x(x+1),\quad w_1(x)=x^2,\quad w_2(x)=x(x-1).
\]
Hence $W(N)$ is the total weight of all paths of $R'$ steps starting at height $x=1$, where an up step from height $x$ has weight $x(x+1)$, a level step $2x^2$, and a down step $x(x-1)$, with an extra factor $2x_{R'+1}$ (twice the final height) when $N$ is even.
\end{theorem}
\begin{proof}
Use \eqref{eq:ferrers}. All small bounds $1,\dots,R'$ are less than $N/2$ and all large bounds $N-R',\dots,N-1$ are greater than $N/2$. So the sorted order consists of the small bounds by increasing $r$, then the midpoint (if $N$ is even), then the large bounds by decreasing $r$.
The $s_r$ rows with small bound $r$ occupy positions $S_{r-1}+1,\dots,S_{r-1}+s_r$, with factors $r-(S_{r-1}+j)+1=x_r-j+1$ for $j=1,\dots,s_r$.
Before the $2-s_r$ rows with large bound $N-r$ come all $S_{R'}$ small-bound rows, the midpoint row ($[N\text{ even}]$ rows), and $\sum_{r'>r}(2-s_{r'})$ large-bound rows. Since $S_{R'}+\sum_{r'>r}(2-s_{r'})=S_r+2(R'-r)$ and $N-2R'-[N\text{ even}]=1$, the $j$th such row has factor
\[
  (N-r)-\bigl(S_r+2(R'-r)+[N\text{ even}]+j\bigr)+1=x_r-s_r+2-j\qquad(j=1,\dots,2-s_r).
\]
Collecting the factors of pair $r$ gives $x_r(x_r+1)$, $x_r\cdot x_r$, or $x_r(x_r-1)$ for $s_r=0,1,2$. The midpoint is in position $S_{R'}+1$, with factor $N/2-S_{R'}=x_{R'+1}$.
This proves the first assertion. Summing over types, the two choices for $s_r=1$ and the two directions of the midpoint each contribute a factor $2$.
\end{proof}
\noindent The height has a direct meaning: The quantity $x_r-1$ is the number of pairs among the first $r-1$ in which both arcs point to the farther end, minus the number of pairs in which both point to the nearer end. A path that reaches height $0$ has weight $0$, so only Motzkin prefixes (paths staying at height $\ge1$) contribute.
For example, $N=5$ gives $R'=2$, and $W(5)=2\cdot(6+8+2)+2\cdot(2+2+0)=40$: The first step is up (weight $2$) or level (weight $2$), and the brackets are the total weights of the second step from heights $2$ and $1$.

\subsection{Exact enumeration of types}\label{ssec:typecount}
\begin{corollary}[Feasible types]\label{cor:feasible}
There are exactly $\binom{N}{\lfloor N/2\rfloor}$ feasible types of order $N$.
\end{corollary}
\begin{proof}
By Theorem~\ref{thm:motzkin}, a type is feasible if and only if $x_r\ge1$ for $1\le r\le R'+1$ (for $s_r=2$, $x_r(x_r-1)>0$ is equivalent to $x_{r+1}=x_r-1\ge1$).
Write $+1$ for each vertex of pair $r$ with the large bound and $-1$ for each vertex with the small bound, and list these $2R'$ signs by increasing $r$, vertex $r$ before vertex $N-r$. The partial sums satisfy $P_{2r}=2r-2S_r=2(x_{r+1}-1)$, so feasibility means $P_{2r}\ge0$ for $1\le r\le R'$. Partial sums with even index are even, so this is equivalent to all partial sums being $\ge-1$, that is, to the walk prefixed by one step $+1$ staying $\ge0$.
There are $\binom p{\lfloor p/2\rfloor}$ walks of length $p$ from $0$ that stay $\ge0$, and they all begin with $+1$. Hence there are $\binom N{R'}$ feasible types when $N=2R'+1$, and $2\binom{N-1}{R'}=\binom N{N/2}$ when $N=2R'+2$, the factor $2$ coming from the two directions of the midpoint.
\end{proof}

\begin{corollary}[The largest type]\label{cor:maxtype}
The types in which every arc points to the farther end ($s_r\equiv0$, the midpoint in either direction) have the most rainbow maps: $R!\,(R+1)!$ if $N=2R+1$, and $(R!)^2$ if $N=2R$.
\end{corollary}
\begin{proof}
The two possible bounds of $k$ are $k$ and $N-k$, and this type takes the larger one; so its board contains, row by row, the board of any type, and the number of rook placements is monotone under inclusion. The value follows from Theorem~\ref{thm:motzkin} with $x_r=r$.
\end{proof}
\noindent For $N=6$ these types are $\RR\RR\LL\LL\LL$ and $\RR\RR\RR\LL\LL$, with $36=(3!)^2$ rainbow maps each; the first is the type of Figure~\ref{fig:ferrers}.

\begin{corollary}[Types with a single rainbow map]\label{cor:unique}
Exactly $2^R$ types have a single rainbow map. In each of these maps every interior vertex points directly to the ground ($f(k)\in\{0,N\}$).
\end{corollary}
\begin{proof}
The factors are positive integers, so their product is $1$ only if each factor is $1$. This forces $s_r=1$ for every pair ($w_0(x)\ge2$; $w_2(x)>0$ needs $x\ge2$, and then $w_2(x)\ge2$; $w_1(x)=1$ needs $x=1$), so $x_r\equiv1$ and the midpoint factor is $1$.
There are $2$ choices per pair and, if $N$ is even, $2$ directions of the midpoint, giving $2^{R'}\cdot2^{[N\text{ even}]}=2^R$ types. The sorted bounds are then exactly $1,\dots,m$, and the only placement is $\delta_k=\ell_k$.
\end{proof}

\begin{corollary}[One arc to the left]\label{cor:oneleft}
Exactly $\lfloor2^N/3\rfloor$ rainbow maps of order $N$ have exactly one arc pointing left.
\end{corollary}
\begin{proof}
Let $k$ be the vertex pointing left. The multiset of bounds is $\{1,\dots,m\}$ with $N-k$ replaced by $k$. If $k<N/2$, the bound $k$ occurs twice and $k+1$ bounds are $\le k$, so the type is infeasible. If $k=N/2$ the bounds are unchanged and there is one rainbow map.
If $k>N/2$, the sorted bounds are $1,\dots,N-k-1,\;N-k+1,\dots,k-1,\;k,\;k,\;k+1,\dots,N-1$; the rows in sorted positions $N-k,\dots,k-1$ have factor $2$ and the others factor $1$, giving $2^{2k-N}$.
Summing, $2+2^3+\dots+2^{N-2}=(2^N-2)/3$ for odd $N$ and $1+2^2+\dots+2^{N-2}=(2^N-1)/3$ for even $N$.
\end{proof}

\subsection{The distribution of one row}\label{ssec:marginal}
Fix a feasible board $\ell=(\ell_1,\dots,\ell_m)$ and take a uniform rook placement on it. For a row $r$ and an integer $x$ let
\[
  u_r(x):=x-\#\{s\ne r:\ell_s\le x\},
\]
the factor that row $r$ would receive in \eqref{eq:ferrers} if its bound were changed to $x$ and it were placed last among the rows with bound $x$. Formula \eqref{eq:ferrers} remains valid for any order of rows that is compatible with the order of the bounds, and it describes a sampling procedure: In such an order, each row chooses uniformly among its remaining lengths.

\begin{lemma}[Hazard rate]\label{lem:hazard}
Let $\ell_r=x+1$ and $u=u_r(x)$. Then $\Pr[\delta_r=x+1]=\frac1{u+1}$; and if $u>0$, then conditioned on $\delta_r\le x$ the placement is uniform on the board with $\ell_r$ changed to $x$.
\end{lemma}
\begin{proof}
Place first the other rows with bound $\le x$, then row $r$, then the remaining rows; this order is compatible with the bounds. When row $r$ is placed, the earlier rows occupy $\#\{s\ne r:\ell_s\le x\}$ lengths, all $\le x$, so row $r$ chooses uniformly among the $u$ remaining lengths in $[1,x]$ and the length $x+1$, whatever the earlier choices were. The placements with $\delta_r\le x$ are exactly the placements of the modified board.
\end{proof}

\begin{theorem}[One-row marginal]\label{thm:marginal}
For every row $r$ and every integer $a$ with $0\le a\le\ell_r$, $\displaystyle\ \Pr[\delta_r\le a]=\prod_{x=a}^{\ell_r-1}\frac{u_r(x)}{u_r(x)+1}$.
In particular, $\Pr[\delta_r=d]$ is constant for $d$ in every interval $(a,b]\subseteq(0,\ell_r]$ that contains no bound of another row.
\end{theorem}
\begin{proof}
If $u_r(x)=0$ for some $x$ with $a\le x<\ell_r$, the rows $s\ne r$ with $\ell_s\le x$ use all the lengths $1,\dots,x$, so $\Pr[\delta_r\le a]=0$ and both sides vanish. Otherwise write $\Pr[\delta_r\le a]=\prod_{x=a}^{\ell_r-1}\Pr[\delta_r\le x\mid\delta_r\le x+1]$, where every event $\delta_r\le x+1$ has positive probability by induction on $\ell_r-x$, and apply Lemma~\ref{lem:hazard} to each factor; by the second assertion of the lemma, the law given $\delta_r\le x+1$ is uniform on the board with $\ell_r$ replaced by $x+1$ and the other bounds unchanged. If no other row has its bound in $(a,b]$, then $u_r(x)$ increases by one at each step on $[a,b]$, the product telescopes, and $\Pr[\delta_r\le x]=\Pr[\delta_r\le b]\,u_r(x)/u_r(b)$ for $a\le x\le b$.
\end{proof}

\begin{corollary}[One row of the largest type]\label{cor:pps}
On the type of Corollary~\ref{cor:maxtype}, let $k$ have rank $\rho$ and recall $R=\lfloor N/2\rfloor$. Then
\[
  \Pr[\delta_k\le d]=\frac{\rho}{N-d},\qquad\Pr[\delta_k=d]=\frac{\rho}{(N-d)(N-d+1)}\qquad(R<d\le N-\rho),
\]
and on an initial range of lengths the probabilities are constant: $\Pr[\delta_k=d]=\frac{\rho}{R(R+1)}$ for $d\le R+1$ if $N=2R+1$, and $\Pr[\delta_k=d]=\frac{\rho}{R^2}$ for $d\le R$ if $N=2R$ and $k\ne R$. If $N=2R$, the midpoint has $\delta_R$ uniform on $\{1,\dots,R\}$.
\end{corollary}
\begin{proof}
Let $N=2R+1$. Each bound $v\in\{R+1,\dots,2R\}$ occurs twice (vertices $v$ and $N-v$). Row $k$ has bound $\ell_k=N-\rho$, so $u_k(x)=x$ for $x\le R$ and $u_k(x)=x-2(x-R)=2R-x$ for $R+1\le x\le \ell_k-1$.
By Theorem~\ref{thm:marginal}, $\Pr[\delta_k\le a]=\prod_{x=a}^{\ell_k-1}\frac{2R-x}{2R-x+1}=\frac{2R-\ell_k+1}{2R-a+1}=\frac\rho{N-a}$ for $R+1\le a\le \ell_k$, and $\Pr[\delta_k\le a]=\frac{\rho}{R+1}\cdot\frac aR$ for $a\le R$.
For $N=2R$ and $k\ne R$, the bounds $R+1,\dots,2R-1$ occur twice and $R$ once (the midpoint), so $u_k(x)=x$ for $x<R$, $u_k(R)=R-1$, and $u_k(x)=2R-1-x$ for $R<x<\ell_k$. Hence $\Pr[\delta_k\le a]=\frac{\rho}{N-a}$ for $R+1\le a\le\ell_k$, then $\Pr[\delta_k\le R]=\frac\rho{R-1}\cdot\frac{R-1}R=\frac\rho R$, and $\Pr[\delta_k\le a]=\frac{\rho a}{R^2}$ for $a\le R$. All bounds other than the midpoint's exceed $R$, so $u_R(x)=x$ and $\delta_R$ is uniform.
\end{proof}
\noindent For example, on the board of Figure~\ref{fig:ferrers} the length $\delta_1$ takes the values $1,\dots,5$ with probabilities $\frac19,\frac19,\frac19,\frac16,\frac12$.

\subsection{Negative dependence}\label{ssec:negdep}
\begin{theorem}[Negative dependence]\label{thm:negdep}
Take a uniform rook placement on a feasible board. Let $I$, $J$ be disjoint sets of rows and fix an integer vector $a_I$ (all thresholds below are integers, and inequalities between vectors are coordinatewise). Then $\Pr[\delta_I\le a_I\mid\delta_J\le y_J]$ is nondecreasing in each $y_j$ (wherever the conditional probability is defined). Consequently
\[
  \Pr[\delta_I\le a_I,\ \delta_J\le b_J]\ \le\ \Pr[\delta_I\le a_I]\,\Pr[\delta_J\le b_J],
\]
and $\mathrm{Cov}\bigl(\phi_1(\delta_i),\phi_2(\delta_j)\bigr)\le0$ for any two rows $i\ne j$ and any nondecreasing functions $\phi_1$, $\phi_2$.
\end{theorem}
\begin{proof}
Conditioning on $\delta_J\le y_J$ gives the uniform measure on the board with $\ell_j$ replaced by $\min(y_j,\ell_j)$ ($j\in J$), again a Ferrers board. For $I=\{i\}$, Theorem~\ref{thm:marginal} expresses the conditional probability as a product of factors $\frac u{u+1}$ with
$u=x-\#\{s\notin J,\,s\ne i:\ell_s\le x\}-\#\{j\in J:\min(y_j,\ell_j)\le x\}$; each $u$ is nondecreasing in each $y_j$, hence so is the product. For larger $I$ use the chain rule: The $p$th factor is the one-row case on the board in which $J$ is cut at $y_J$ and $i_1,\dots,i_{p-1}$ are cut at $a$. When comparing $y_J\le y'_J$ we may assume that the conditional probability at $y_J$ is positive (otherwise there is nothing to prove); then every intermediate conditioning event has positive probability at $y_J$, hence also at $y'_J$, and the one-row case applies factor by factor.
Feasibility is preserved when $y$ increases (Hall's condition becomes weaker), so the domain is an up-set. The inequality compares $y_J=b_J$ with $y_J=\ell_J$. For two rows this is negative quadrant dependence, and Hoeffding's identity
$\mathrm{Cov}(\phi_1(X),\phi_2(Y))=\sum_{a,b}\Delta\phi_1(a)\Delta\phi_2(b)\bigl[\Pr(X\le a,Y\le b)-\Pr(X\le a)\Pr(Y\le b)\bigr]$ with $\Delta\phi(a):=\phi(a+1)-\phi(a)$, gives the last assertion.
\end{proof}
\noindent Equality can also be described. Call $t$ a \emph{full column} if $\#\{s:\ell_s\le t\}=t$; full columns cut the board into independent blocks. The covariance of $\delta_i$ and $\delta_j$ vanishes if and only if one of the two rows is constant or some full column $t$ satisfies $\min(\ell_i,\ell_j)\le t<\max(\ell_i,\ell_j)$; otherwise it is strictly negative. (For strictness, let $\ell_i\le\ell_j$, $a:=\ell_i-1$ and $U:=u_i(a)$; then $U\ge1$ since row $i$ is not constant. By Lemma~\ref{lem:hazard}, $\Pr[\delta_i\le a]=\frac U{U+1}$, while on the board with $\ell_j$ cut to $a$ the same lemma gives $\Pr[\delta_i\le a\mid\delta_j\le a]=\frac{U-1}U$. Here $\Pr[\delta_j\le a]>0$, since otherwise there would be a full column $t$ with $\ell_i-1\le t<\ell_j$, and $t=\ell_i-1$ would make row $i$ constant. So the term $(a,a)$ of Hoeffding's sum is strictly negative, and the other terms are $\le0$.)

For the uniform measure on all rainbow maps, a mixture over types, the quadrant form of Theorem~\ref{thm:negdep} fails: For $N=6$, $\Pr[\delta_1\le3,\delta_3\le2]=\frac16>\Pr[\delta_1\le3]\Pr[\delta_3\le2]=\frac{143}{882}$. Exact computation shows $\mathrm{Cov}(\delta_i,\delta_j)<0$ for every pair $i<j$ when $N\le13$; we have no proof in general.

\subsection{The number of arcs pointing left}\label{ssec:left}
\begin{theorem}\label{thm:left}
Choose $f$ uniformly among the $W(N)$ rainbow maps of order $N$; let $L(f)$ be the number of arcs pointing left and $\kappa$ the number of level steps of its path (Section~\ref{ssec:pairs}), so that paths are weighted as in Theorem~\ref{thm:motzkin}. Given the path $(s_1,\dots,s_{R'})$,
\[
  L=(R'-\kappa)+2B+[N\text{ even}]\,B',\qquad B\sim\mathrm{Bin}(\kappa,\tfrac12),\quad B'\sim\mathrm{Bernoulli}(\tfrac12),\quad\text{independent}.
\]
Hence $\Ex L=\frac{N-1}2$ and $\mathrm{Var}\,L=\Ex\kappa+\frac14[N\text{ even}]$.
\end{theorem}
\begin{proof}
A pair with $s_r=0$ is $(r,\RR),(N-r,\LL)$ and a pair with $s_r=2$ is $(r,\LL),(N-r,\RR)$; each has one arc pointing left. A level pair is $(r,\LL),(N-r,\LL)$ or $(r,\RR),(N-r,\RR)$, with two or zero arcs pointing left.
Given the path, all types have the same multiset of bounds and hence, by Theorem~\ref{thm:motzkin}, the same number of rainbow maps; so the choices in the level pairs and at the midpoint are independent fair coins. Since $\Ex[L\mid\text{path}]=\frac{N-1}2$ is constant, the variance follows from the law of total variance.
\end{proof}

\begin{corollary}\label{cor:leftgf}
For $N\ge2$, $\sum_fz^{L(f)}$ is a palindromic polynomial of degree $N-1$ that vanishes at $z=-1$. The number $A_2(N)$ of rainbow maps with exactly two arcs pointing left is given by
\[
  90\,A_2(N)=\begin{cases}3^{N+2}-16\cdot2^N-15N+20,&N\text{ odd},\\ 3^{N+2}-14\cdot2^N-15N+5,&N\text{ even}.\end{cases}
\]
\end{corollary}
\begin{proof}
By Theorem~\ref{thm:left}, the generating function is $z^{R'}(1+z)^{[N\text{ even}]}$ times the total weight of the paths in which the level weight $2x^2$ is replaced by $(z+z^{-1})x^2$ and the final factor $2x_{R'+1}$ by $x_{R'+1}$; so it is palindromic. It vanishes at $z=-1$ because of an involution that changes the parity of $L$: If $k$ is the vertex of length $1$, replace $f(k)=k+1$ by $f(k)=k-1$ and vice versa (both targets lie in $[0,N]$, and all lengths are unchanged).
For $A_2$, the first step that is not level must be up, since a down step from height $1$ has weight $0$. For $N=2R'+1$, either all steps are level (height $1$ throughout) and exactly one level pair has two left arcs ($R'$ ways), or there are exactly two non-level steps $r_1<r_2$: The first is up (weight $1\cdot2$), the steps between are level at height $2$ (weight $4$ each), and the second is up ($2\cdot3$, then $9$ per step) or down ($2\cdot1$, then $1$ per step). Thus
$A_2=R'+\sum_{r_1<r_2}2\cdot4^{r_2-r_1-1}\bigl(6\cdot9^{R'-r_2}+2\bigr)$.
For $N=2R'+2$ the midpoint contributes the factor $x_{R'+1}$ if it points right and one more left arc if it points left, so
$A_2=R'+\sum_{r_1<r_2}2\cdot4^{r_2-r_1-1}\bigl(18\cdot9^{R'-r_2}+2\bigr)+\sum_{r=1}^{R'}4\cdot4^{R'-r}$,
the last sum counting the paths with a single non-level step $r$ (up) and the midpoint pointing left. Summing the geometric series gives the formulas.
\end{proof}

\subsection{Vertices that reach the ground}\label{ssec:ground}
How few interior vertices can have their orbit $k\to f(k)\to\cdots$ reach the ground, the others all flowing into cycles?
\begin{theorem}\label{thm:ground}
For $N\ge3$ every rainbow map has at least two interior vertices that reach the ground. For $N\ge6$ this is attained by
\[
  f(1)=N,\qquad f(k)=N-1\quad(2\le k\le N-4),
\]
\[
  f(N-3)=N-4,\qquad f(N-2)=0,\qquad f(N-1)=N-3,
\]
in which only $1$ and $N-2$ reach the ground. For $N=3,4,5$ the minimum is $N-1$.
\end{theorem}
\begin{proof}
The map $f\mapsto\bar f$, $\bar f(k):=N-f(N-k)$, is a bijection on rainbow maps of order $N$ that sends the vertices reaching the ground to their mirror images; so we may assume that the arc of length $N-1$ is $1\to N$ (the only alternative is $N-1\to0$). The arc of length $N-2$ is one of $1\to N-1$, $2\to N$, $N-2\to0$, $N-1\to1$. An arc starting at $1$ is impossible, since vertex $1$ already has its arc; in each remaining case the starting vertex differs from $1$ and reaches the ground ($N-1\to1\to N$).
In the construction the lengths are $N-1$ (vertex $1$), $N-1-k$ for $2\le k\le N-4$ (giving $3,\dots,N-3$), then $1$, $N-2$, $2$. Since $N-4\ge2$, we have $N-4\to N-1$, so $N-1\to N-3\to N-4\to N-1$ is a triangle, and every vertex in $[2,N-4]$ points into it.
For $N\le5$ we have $W(N)=\tau(N)$ (Table~\ref{tab:values}), so there are no cycles.
\end{proof}
\noindent The bijection $f\mapsto\bar f$ has no fixed points for $N\ge3$ (a fixed point would need $\delta_k=\delta_{N-k}$). Hence $W(N)$ is even, and for each $d$ the arc of length $d$ points left in exactly half of the rainbow maps.

\subsection{Graceful trees in which vertex $0$ has large degree}\label{ssec:trees}
\begin{proposition}\label{prop:deg0}
Let $\deg0$ be the degree of vertex $0$ (including the edge $\{0,N\}$). There is $1$ graceful tree with $\deg0=N$, there are $\lfloor N/2\rfloor$ with $\deg0=N-1$, and for $N\ge2$ there are $\frac58N(N-2)$ ($N$ even) or $\frac18(5N-7)(N-1)$ ($N$ odd) with $\deg0=N-2$.
\end{proposition}
\begin{proof}
If $\deg0=N-1$, vertex $0$ misses exactly one $v$; the edge of length $v$ avoids $0$, so it is $\{v,2v\}$ with $2v\le N$.
If $\deg0=N-2$, vertex $0$ misses $x<y$ ($y<N$), and the two remaining edges have lengths $x$ and $y$. If $\{x,y\}$ is not an edge, then $x$ and $y$ hang on neighbors of $0$: If the edge at $x$ has length $x$ and the edge at $y$ has length $y$, the neighbors are $2x$ and $2y$ (requiring $2y\le N$ and $2x\ne y$); if the lengths are exchanged, the neighbors are $x+y$ and $y\pm x$ (requiring $x+y\le N$, and $y-x$ only if $y\ne2x$).
If $\{x,y\}$ is an edge, its length $y-x$ must be $x$, so $y=2x$, and the other edge has length $2x$ and joins $x$ to $3x$ or $y$ to $4x$. Each case gives $N$ edges connected through $0$, hence a graceful tree, and they are distinct.
With $R=\lfloor N/2\rfloor$ and $P=\#\{1\le x<y:x+y\le N\}$ ($R(R-1)$ for $N=2R$, $R^2$ for $N=2R+1$), the three cases give $\binom R2-\lfloor\frac R2\rfloor$, $2P-\lfloor\frac N3\rfloor$ and $\lfloor\frac N3\rfloor+\lfloor\frac N4\rfloor$ trees; since $\lfloor\frac R2\rfloor=\lfloor\frac N4\rfloor$, the total is $\binom R2+2P$.
\end{proof}

%======================================================================
\section{Few cycles}\label{sec:cycles}
This section and the next two prove the right half of Theorem~\ref{thm:compare}. There are three steps.
\emph{Few cycles} (this section): To close a cycle, its last arc must point exactly back to its first vertex $v$, which requires one specific length. We show that, whatever has been revealed so far (as long as it is not much), the probability that a given vertex points to $v$ is at most about $1/\rho(v)$. Summing over vertices and cycle sizes gives $O(\log^2N)$ cycles on average.
\emph{Cycles are easy to repair} (Section~\ref{sec:repair}): Each cycle can be destroyed by changing $O(\log N)$ edges.
\emph{Counting} (Section~\ref{sec:count}): Few changes and few cycles mean that rainbow maps outnumber trees by at most $e^{O(\log^4N)}$.

\subsection{Longer lengths are not harder to get}\label{ssec:swap}
Fix a feasible type $\sigma$ and take the uniform measure on its board. Then the target $f(k)$ is determined by the length $\delta_k$ ($k-\delta_k$ if $\sigma_k=\LL$, $k+\delta_k$ if $\sigma_k=\RR$).

A row--length pair is a \emph{cell}. A \emph{history} is a set of $n$ cells in distinct rows and with distinct lengths, consistent with at least one rainbow map of type $\sigma$; it specifies the lengths of $n$ rows, which are called \emph{revealed}. Given a history, the rest is uniform over its \emph{completions}. Let $A$ be the set of lengths not used by the history.

\begin{lemma}[Swap]\label{lem:swap}
Given a history, let row $k$ be unrevealed and let $d<d'$ be lengths in $A$ with $d'\le\ell_k$. Then $\#\{\text{completions with }\delta_k=d\}\le\#\{\text{completions with }\delta_k=d'\}$.
Hence, for $d\in A$ with $d\le\ell_k$, $\ \Pr[\delta_k=d\mid\text{history}]\le1\big/\bigl|A\cap[d,\ell_k]\bigr|$.
\end{lemma}
\begin{proof}
Take a completion with $\delta_k=d$. The length $d'$ is not used by the history, so the row $k'$ using it is unrevealed, and $k'\ne k$. Swap the lengths of $k$ and $k'$: Row $k$ gets $d'\le\ell_k$, and row $k'$ gets $d<d'\le\ell_{k'}$, which is allowed because the allowed set of $k'$ is a prefix. This is a completion with $\delta_k=d'$; swapping $k$ with the row that uses $d$ recovers the original, so the map is injective.
For the second assertion, each length in $A\cap[d,\ell_k]$ is at least as likely for $k$ as $d$, and these probabilities sum to at most $1$.
\end{proof}
\noindent The lemma uses only that the allowed sets are prefixes. For example, on the board of Figure~\ref{fig:ferrers} with the empty history, the probabilities of the lengths of vertex $1$ are nondecreasing (Corollary~\ref{cor:pps}), and the swap with $d'=5$ sends the placement $(3,4,1,2,5)$ of the figure to $(5,4,1,2,3)$.

\begin{lemma}[Pointing to a given vertex]\label{lem:point}
Given a history with $n$ cells, let row $k$ be unrevealed and let $v$ be an interior vertex with $n\le\rho(v)$. Then $\ \Pr[k\to v\mid\text{history}]\le\dfrac1{\rho(v)+1-n}$.
\end{lemma}
\begin{proof}
The probability is zero if $v$ is not in the direction of $k$ or the required length $d=|k-v|$ is used. Otherwise, if $\sigma_k=\LL$ then $d=k-v$ and $\ell_k=k$, so $[d,\ell_k]$ contains $v+1$ lengths; if $\sigma_k=\RR$ then $d=v-k$ and $\ell_k=N-k$, so it contains $N-v+1$ lengths. Either way there are at least $\rho(v)+1$, of which at most $n$ are used, so $|A\cap[d,\ell_k]|\ge\rho(v)+1-n$. Apply Lemma~\ref{lem:swap}.
\end{proof}
\noindent The bound is uniform in $k$ and in the history, and it is attained ($N=3$, type $\RR\LL$: $\Pr[1\to2]=\frac12$). It bounds the probabilities without ordering them: Given $\delta_3=1$ on the board of Figure~\ref{fig:ferrers}, the bounds for $v=2$ and $v=3$ are $\frac12>\frac13$, but $\Pr[1\to2]=0<\Pr[1\to3]=\frac16$.

\subsection{The expected number of cycles}\label{ssec:cycles}
\begin{lemma}\label{lem:memb}
Let $t\ge3$ and $\rho(v)\ge t-1$. Then $\ \Pr[v\text{ lies on a }t\text{-cycle}\mid\sigma]\le\dfrac1{\rho(v)-t+2}$, which is at most $\dfrac2{\rho(v)}$ if moreover $\rho(v)\ge2t$.
\end{lemma}
\begin{proof}
Reveal the forward path from $v=u_0$: Revealing the length of $u_0$ gives $u_1$, revealing the length of $u_1$ gives $u_2$, and so on, $t-1$ arcs in all, ending at $u_{t-1}$. \emph{Stop} as soon as some $u_j$ ($1\le j\le t-1$) is a ground vertex or equals an earlier $u_i$; then $v$ is not on a $t$-cycle.
Each outcome that does not stop is a history with $n=t-1$ cells (the statement that these $t-1$ rows have these lengths, with no further condition), in which $u_0,\dots,u_{t-1}$ are distinct and $u_{t-1}$ is unrevealed. The vertex $v$ lies on a $t$-cycle if and only if the process does not stop and $u_{t-1}\to v$.
Apply Lemma~\ref{lem:point} to each such history ($n=t-1\le\rho(v)$) and average. For the second bound, $\rho(v)\ge2t$ gives $\rho(v)-t+2\ge\rho(v)/2$.
\end{proof}

The bound of Lemma~\ref{lem:memb} is useless for vertices of small rank, since after $t-1$ revealed cells few lengths may remain. We therefore split cycles into two kinds. Cycles most of whose vertices have large rank are handled by summing Lemma~\ref{lem:memb} over vertices. Cycles most of whose vertices have rank $<2t$ need no probability at all: There are fewer than $4t$ such vertices, and distinct cycles are disjoint.

\begin{theorem}[Number of cycles]\label{thm:cycles}
For every feasible type $\sigma$,
\[
  \Ex[c\mid\sigma]\ \le\ B_N:=8\mathcal H_R^2+16\lfloor\log_2m\rfloor,\qquad \mathcal H_R=\sum_{j=1}^{R}\frac1j .
\]
Averaging over $\sigma$, $\Ex[c]\le B_N=O(\log^2N)$.
\end{theorem}
\begin{proof}
A $t$-cycle is of the \emph{first kind} if at least $t/2$ of its vertices have $\rho\ge2t$, and of the \emph{second kind} otherwise (then more than $t/2$ of its vertices have $\rho<2t$).

\emph{First kind.} Let $V_t:=\{v:\rho(v)\ge2t$ and $v$ lies on a $t$-cycle$\}$. Each $t$-cycle of the first kind contributes at least $t/2$ elements of $V_t$, and distinct cycles contribute distinct elements, so the number $c^{\mathrm I}_t$ of such cycles satisfies $\frac t2\,c^{\mathrm I}_t\le|V_t|$. By Lemma~\ref{lem:memb}, $\Ex[|V_t|\mid\sigma]\le\sum_{v:\rho(v)\ge2t}2/\rho(v)$; since each rank has at most two vertices,
\[
  \Ex[c^{\mathrm I}_t\mid\sigma]\le\frac2t\sum_{v:\rho(v)\ge2t}\frac2{\rho(v)}\le\frac8t\sum_{j=2t}^{R}\frac1j,
  \qquad\sum_{t\ge3}\Ex[c^{\mathrm I}_t\mid\sigma]\le8\Bigl(\sum_{t\le R/2}\frac1t\Bigr)\mathcal H_R\le8\mathcal H_R^2 .
\]
\emph{Second kind (a deterministic bound).} Fix a dyadic range $2^j\le t<2^{j+1}$; since $3\le t\le m$, we have $1\le j\le\lfloor\log_2m\rfloor$. Each second-kind cycle in this range uses more than $t/2\ge2^{j-1}$ vertices with $\rho<2t<2^{j+2}$, and there are fewer than $2\cdot2^{j+2}=2^{j+3}$ such vertices. Cycles are disjoint, so this range contains fewer than $2^{j+3}/2^{j-1}=16$ cycles of the second kind.
Adding, $\Ex[c\mid\sigma]\le8\mathcal H_R^2+16\lfloor\log_2m\rfloor$.
\end{proof}
\noindent The proof uses no typical property of $\sigma$. The bound is not meant for small $N$: For $N=1000$, $B_N\approx513$ exceeds the trivial bound $c\le m/3$; for $N=10^6$, $B_N\approx1805$.

%======================================================================
\section{Cycles are easy to repair}\label{sec:repair}
\subsection{What has to be done}\label{ssec:todo}
In this section directions play no role: A rainbow map $f$ enters only through its undirected edge set $X(f):=\bigl\{\{k,f(k)\}:1\le k\le m\bigr\}$, and the type is irrelevant. These $m$ edges are distinct (two vertices choosing the same edge would use the same length), and each length $1,\dots,m$ occurs exactly once.

We work in the \emph{contracted graph}: Take the complete graph $K_{N+1}$ on $\{0,\dots,N\}$, delete the edge $\{0,N\}$, and identify $0$ and $N$ into one vertex $g$. Its vertices are $g,1,\dots,m$, and each edge keeps its identity and its length from $K_{N+1}$. It is a multigraph: The edges $\{0,k\}$ and $\{k,N\}$ are two different edges between $g$ and $k$ (of lengths $k$ and $N-k$), and two parallel edges form a cycle. Let $E$ be its edge set; $|E|=\sum_{d=1}^m(N+1-d)=\frac{(N-1)(N+2)}2$.
In the contracted graph the component of $X(f)$ containing $g$ is a tree and every other component contains exactly one cycle; these are the cycles of $f$. Graceful labelings correspond to \emph{rainbow spanning trees} of the contracted graph, that is, spanning trees with exactly one edge of each length: Deleting $\{0,N\}$ from a graceful tree and contracting gives such a tree, and conversely a rainbow spanning tree, lifted to $K_{N+1}$ together with $\{0,N\}$, has $N$ edges on $N+1$ vertices, is connected, and has lengths $1,\dots,N$. So there are $\tau(N)$ rainbow spanning trees.

Given a rainbow map with cycles, we want to change its edge set into a rainbow spanning tree by changing few edges. Deleting one edge from each cycle leaves a forest in which some lengths are missing. So the problem is this: \emph{Given a forest $F$ with at most one edge of each length and some length missing, how do we add an edge of a missing length?}
If some edge of a missing length joins two components of $F$, we simply add it. Otherwise every edge of that length closes a cycle in $F$; we can add one and delete another edge of that cycle, of some length $d'$, so that $d'$ is now missing, and continue. We must show that such chains are short.

For this we need to know: \emph{How many edges without a cycle can be chosen using only $k$ distinct lengths?} The answer is at least $\min(2k,m)$ (Corollary~\ref{cor:rank}). Then the set of lengths involved in a chain at least doubles at each level, so $\lceil\log_2m\rceil$ levels suffice (Section~\ref{ssec:augment}).

\subsection{How large a forest can $k$ lengths support?}\label{ssec:word}
For $\Delta\subseteq\{1,\dots,m\}$ let $G_\Delta$ be the graph on $\{0,\dots,N\}$ consisting of all edges with lengths in $\Delta\cup\{N\}$ (the only edge of length $N$ is $\{0,N\}$), and let $r(\Delta)$ be the number of edges of a spanning forest (a maximal acyclic subgraph) of $G_\Delta$, that is, $N+1$ minus the number of components.

The key observation is that an edge of length $d$ joins $i$ and $i+d$. So if we label each vertex $i$ by the name of its component in $G_\Delta$, the resulting word $w=w_0\ldots w_N$ of length $N+1$ has every $d\in\Delta\cup\{N\}$ as a period. The number of lengths becomes a number of periods, and the number of components becomes the number of distinct letters. The following inequality says that many periods force few letters.

A \emph{border} of a word $w$ is a nonempty proper prefix that is also a suffix. Let $|w|$ be the length, $\mathrm{sym}(w)$ the number of distinct letters, and $\mathrm{bor}(w)$ the number of borders. For $1\le p<|w|$, $w$ has a border of length $|w|-p$ if and only if $p$ is a period of $w$ ($w_i=w_{i+p}$ for $0\le i<|w|-p$). For example, the borders of $abaaba$ are $a$ and $aba$ (periods $5$ and $3$).

\begin{theorem}[Border bound]\label{thm:border}
If $w$ is not constant, then $\ |w|-\mathrm{sym}(w)\ge2\,\mathrm{bor}(w)-1$. (Equality holds for $a^kba^k$ and $a^kbca^k$ with $k\ge1$ and distinct letters $a$, $b$, $c$.)
\end{theorem}
\begin{proof}
Induction on $|w|$. If $\mathrm{bor}(w)=0$ then $|w|-\mathrm{sym}(w)\ge0>-1$. Let $\mathrm{bor}(w)\ge1$ and let $z$ be the longest border of $w$. Every shorter border of $w$ is a prefix and a suffix of $z$, hence a border of $z$, and conversely; so $\mathrm{bor}(w)=1+\mathrm{bor}(z)$. Let $q:=|z|$; then $p:=|w|-q$ is a period of $w$, and $p\ge2$ because $p=1$ would make $w$ constant.

\emph{$z$ is not constant.} Then $q\ge2$, and by induction $q-\mathrm{sym}(z)\ge2\,\mathrm{bor}(z)-1=2\,\mathrm{bor}(w)-3$. Since $w$ has period $p$, every letter of $w$ occurs among its first $p$ positions; a letter not in $z=w_0\ldots w_{q-1}$ must occur in positions $q,\dots,p-1$, so $\mathrm{sym}(w)-\mathrm{sym}(z)\le\max(0,p-q)$. Hence
\[
  \bigl(|w|-\mathrm{sym}(w)\bigr)-\bigl(q-\mathrm{sym}(z)\bigr)=p-\bigl(\mathrm{sym}(w)-\mathrm{sym}(z)\bigr)\ge\min(p,q)\ge2 ,
\]
and $|w|-\mathrm{sym}(w)\ge(2\,\mathrm{bor}(w)-3)+2=2\,\mathrm{bor}(w)-1$.

\emph{$z=a^q$.} Then $\mathrm{bor}(z)=q-1$ and $\mathrm{bor}(w)=q$. The word $w$ begins and ends with $a^q$; if $|w|\le2q$ it would be constant, so $|w|\ge2q+1$, and the letters other than $a$ lie in the middle $|w|-2q$ positions. Thus $\mathrm{sym}(w)\le1+|w|-2q$, that is, $|w|-\mathrm{sym}(w)\ge2q-1$.
\end{proof}

\begin{lemma}[$k$ lengths and the length $N$ support a forest of $\min(2k+1,N)$ edges]\label{lem:rank}
In the complete graph $K_{N+1}$, $\ r(\Delta)\ge\min\bigl(2|\Delta|+1,\,N\bigr)$ for every $\Delta\subseteq\{1,\dots,m\}$.
\end{lemma}
\begin{proof}
Let $k=|\Delta|$ and let $w=w_0\ldots w_N$ be the word of component names, so $r(\Delta)=(N+1)-\mathrm{sym}(w)$. Each $d\in\Delta$ is a period of $w$ and gives a border of length $N+1-d\in\{2,\dots,N\}$; the edge $\{0,N\}$ gives $w_0=w_N$, a border of length $1$. These $k+1$ borders have distinct lengths, so $\mathrm{bor}(w)\ge k+1$.
If $w$ is not constant, Theorem~\ref{thm:border} gives $r(\Delta)\ge2(k+1)-1=2k+1$. If $w$ is constant, $G_\Delta$ is connected and $r(\Delta)=N$.
\end{proof}

\begin{example}\label{ex:rank}
For $N=6$ and $\Delta=\{2,4\}$ the even vertices form one component and the odd vertices another; $w=ABABABA$, with borders $A$, $ABA$, $ABABA$. Here $\mathrm{sym}(w)=2$, $\mathrm{bor}(w)=3$, and $r(\{2,4\})=5=\min(2\cdot2+1,6)$: Both Theorem~\ref{thm:border} and Lemma~\ref{lem:rank} hold with equality.
\end{example}

\begin{corollary}\label{cor:rank}
In the contracted graph, at least $\min(2|\Delta|,m)$ edges with lengths in $\Delta$ can be chosen without forming a cycle.
\end{corollary}
\begin{proof}
The edges with lengths in $\Delta$ are the edges of $G_\Delta$ with $\{0,N\}$ removed and $0$, $N$ identified. Since $0$ and $N$ are joined in $G_\Delta$, the identification does not change the number of components and removes one vertex; so its spanning forests (maximal acyclic subgraphs) have $r(\Delta)-1$ edges. Apply Lemma~\ref{lem:rank}.
\end{proof}

\subsection{Augmentation}\label{ssec:augment}
Recall $D=\lceil\log_2m\rceil$. Below, $F$ is always a forest of the contracted graph (an acyclic set of edges, not necessarily maximal) with at most one edge of each length. A length $d$ is \emph{missing} if $F$ has no edge of length $d$, and \emph{directly insertable} if some edge of length $d$ not in $F$ joins two components of $F$.
For an edge $e\notin F$ of length $d$ whose ends lie in the same component of $F$, let $P_F(e)$ be the path in $F$ joining its ends; each edge of length $d'$ on this path gives a \emph{relation} $d\to d'$.

\begin{lemma}[Augmentation]\label{lem:augment}
Suppose $|F|<m$. Then there exist $s\le D$, edges $e_0,\dots,e_s\notin F$ and edges $f_1,\dots,f_s\in F$ such that
\[
  F':=\bigl(F\cup\{e_0,\dots,e_s\}\bigr)\setminus\{f_1,\dots,f_s\}
\]
is a forest with at most one edge of each length and $|F'|=|F|+1$.
\end{lemma}
\begin{proof}
\emph{Finding a chain.} Search by levels: Level $0$ consists of the missing lengths (there is at least one since $|F|<m$), and level $j+1$ consists of the lengths not seen before that are reached from level $j$ by one relation. Let $\Delta_j$ be the union of levels $0,\dots,j$, and $F_j$ the set of edges of $F$ with lengths in $\Delta_j$.

Suppose $\Delta_j$ contains no directly insertable length. Then both ends of every edge with length in $\Delta_j$ lie in one component of $F_{j+1}$: For edges in $F$ this is clear, and for $e\notin F$ the ends lie in one component of $F$ and all lengths on $P_F(e)$ belong to $\Delta_{j+1}$. So an acyclic set of edges with lengths in $\Delta_j$ has at most $n_C-1$ edges inside each component $C$ of $F_{j+1}$ with $n_C$ vertices, hence at most $|F_{j+1}|\le|\Delta_{j+1}|$ edges in all. By Corollary~\ref{cor:rank} such a set can have $\min(2|\Delta_j|,m)$ edges; therefore
\[
  |\Delta_{j+1}|\ \ge\ \min(2|\Delta_j|,\,m).
\]
If none of $\Delta_0,\dots,\Delta_{D-1}$ contained a directly insertable length, then $|\Delta_D|\ge\min(2^D,m)=m$, so $\Delta_D$ would contain all lengths. But $F$ is disconnected (as $|F|<m$) while the contracted graph is connected, so some edge joins two components of $F$; its length is directly insertable and lies in $\Delta_D$.
So some $\Delta_j$ with $j\le D$ contains a directly insertable length. Take a shortest chain of relations $d_0\to d_1\to\dots\to d_s$ from a missing length $d_0$ to a directly insertable length $d_s$. Then $s\le D$, the $d_i$ are distinct, and for $0\le i<s$ there is an edge $e_i\notin F$ of length $d_i$ such that $P_F(e_i)$ contains an edge $f_{i+1}\in F$ of length $d_{i+1}$; and $e_s\notin F$ is an edge of length $d_s$ joining two components of $F$. As the chain is shortest, $f_j\notin P_F(e_i)$ for $j>i+1$.

\emph{The result.} $F'$ loses one edge of each length $d_1,\dots,d_s$ and gains one edge of each length $d_0,\dots,d_s$; so only the missing length $d_0$ is added, each length still occurs at most once, and $|F'|=|F|+1$.
To see that $F'$ is acyclic, let $F^{(j)}:=\bigl(F\cup\{e_j,\dots,e_s\}\bigr)\setminus\{f_{j+1},\dots,f_s\}$ for $j=s,s-1,\dots,0$, so that $F^{(s)}=F\cup\{e_s\}$ and $F^{(0)}=F'$. $F^{(s)}$ is a forest because $e_s$ joins two components of $F$.
If $F^{(j+1)}$ is a forest ($j<s$), the path $P_F(e_j)$ lies in $F$ and avoids $f_{j+2},\dots,f_s$, so it lies in $F^{(j+1)}$ and contains $f_{j+1}$; it is the unique path in $F^{(j+1)}$ joining the ends of $e_j$. Replacing $f_{j+1}$ by $e_j$ on this path leaves a forest $F^{(j)}$. By induction $F'$ is a forest. (The intermediate sets $F^{(j)}$ may contain two edges of the same length; this does not matter.)
\end{proof}

\begin{lemma}[Repair]\label{lem:repair}
For every rainbow map $f$ there is a rainbow spanning tree $Y$ with $\ |X(f)\triangle Y|\le2(D+1)\,c(f)$.
\end{lemma}
\begin{proof}
Delete an arbitrary edge from each cycle; this gives a forest $F_0$ with at most one edge of each length and $|F_0|=m-c$. Applying Lemma~\ref{lem:augment} $c$ times gives a forest with $m$ edges, one of each length, that is, a rainbow spanning tree $Y$. The deletions change $c$ edges and each augmentation changes at most $2D+1$, so $|X(f)\triangle Y|\le c+c(2D+1)=2(D+1)c$.
\end{proof}
\noindent Two remarks. First, we only need such a $Y$ to \emph{exist}; any choice of $Y(f)$ works in Section~\ref{sec:count}, and no explicit construction is needed.
Second, the repair does not proceed one cycle at a time. Each augmentation adds \emph{some} missing length, and the set of missing lengths shrinks by exactly one; but its exchanges may replace other edges by edges of the same length, including edges that were on no cycle and edges that an earlier augmentation added. Only the total number of changes is counted.

\begin{example}\label{ex:repair6}
For the rainbow map of Figure~\ref{fig:config}(b), with cycle $1\to3\to4\to1$, delete the edge $\{3,4\}$ of length $1$. The components are $\{1,3,4\}$ and $\{g,2,5\}$, and the edge $\{0,1\}$ of length $1$ joins them. Adding it gives the tree of Figure~\ref{fig:config}(a), with two edges changed.
\end{example}

%======================================================================
\section{Counting: proof of Theorem~\ref{thm:compare}}\label{sec:count}
For each rainbow map $f$ choose a rainbow spanning tree $Y(f)$ as in Lemma~\ref{lem:repair}, so that $Z(f):=X(f)\triangle Y(f)$ is small. Different maps may give the same tree, but $f$ can be recovered from the tree together with the set $Z(f)$ and the orientation of each cycle of $f$. The number of choices of $Z$ and of orientations is $e^{O(\log^4N)}$, since $Z$ is small and cycles are few. Since Theorem~\ref{thm:cycles} controls only the average number of cycles, we discard the maps with more than $2B_N$ cycles, which by Markov's inequality form less than half of all maps.

\begin{proof}[Proof of Theorem~\ref{thm:compare}]
The left inequality is Lemma~\ref{lem:inj}. Let $c_\ast:=\lfloor2B_N\rfloor$ and $\mathcal G:=\{f:c(f)\le c_\ast\}$. By Theorem~\ref{thm:cycles} and Markov's inequality, $\#\{f:c(f)\ge c_\ast+1\}\le W(N)\,B_N/(c_\ast+1)<\frac12W(N)$, so $|\mathcal G|>\frac12W(N)$.

For $f\in\mathcal G$ we have $Z(f)\subseteq E$ and $|Z(f)|\le Q:=2(D+1)c_\ast$. \emph{The map $f$ is determined by $Y(f)$, $Z(f)$ and an orientation of each cycle:} The set $Y(f)\triangle Z(f)=X(f)$ is the edge set of $f$ (as a subset of $E$, so edges keep their identities). Each interior vertex leaves along exactly one of these edges and distinct vertices use distinct edges, so in the component containing $g$ every edge is oriented towards $g$, and in every other component the edges off the cycle are oriented towards the cycle; only the orientation of each cycle remains. An edge together with its orientation determines the vertex $k$ it leaves and $f(k)$.
There are at most $\tau(N)$ choices for $Y(f)$, at most $\sum_{j\le Q}\binom{|E|}j\le(|E|+1)^{Q}$ choices for $Z(f)$ (write a set of at most $Q$ edges as a sequence of length $Q$, padded with a blank symbol), and at most $2^{c_\ast}$ orientations. Hence
\[
  \tfrac12W(N)\ <\ |\mathcal G|\ \le\ \tau(N)\,(|E|+1)^{Q}\,2^{c_\ast},
\]
and, using $c_\ast\le2B_N$,
\[
  \log W(N)-\log\tau(N)\ \le\ \bigl[4(D+1)\log(|E|+1)+2\log2\bigr]B_N+\log2 .
\]
\emph{The constant.} Note that $\log N\le\Lambda-1$. From $D+1\le\frac{\log N}{\log2}+2$ and $|E|+1=\frac{N(N+1)}2\le N^2$, the bracket is at most $\frac8{\log2}\log^2N+16\log N+1.39\le11.55\Lambda^2$.
From $\mathcal H_R\le1+\log\frac N2\le\Lambda-0.69$ and $\lfloor\log_2m\rfloor\le\frac{\log N}{\log2}$ we get $B_N\le8(\Lambda-0.69)^2+23.09(\Lambda-1)\le8\Lambda^2+12.05\Lambda$.
The product is at most $92.4\Lambda^4+139.2\Lambda^3$. For $N\ge3$, $\Lambda\ge2.38$, so $139.2\Lambda^3\le58.5\Lambda^4$ and $\log2\le0.03\Lambda^4$; the total is at most $151\Lambda^4$.
\end{proof}
\noindent The four logarithms in $\log^4N$ have separate origins: two in $B_N$ (one of them from summing $1/t$ over the sizes $t$ of cycles), one from the $2(D+1)$ changes per cycle, and one from the $\log|E|$ needed to name an edge.

%======================================================================
\section{Formulas for the Whitty numbers}\label{sec:perm}
This section does not involve trees; it determines the size of $W(N)$ and proves Theorem~\ref{thm:perm}, starting from the Motzkin paths of Section~\ref{sec:struct}.
We write the path transitions as a difference operator $T$, so that $W(N)$ is the value of a power of $T$ at $x=1$ (Section~\ref{ssec:operator}). For a special family of polynomials $p_k$, the value of $p_k(T)$ at $x=1$ is a positive lattice path count that turns out to be a factorial (Section~\ref{ssec:bridge}). Expanding powers of $T$ in the $p_k(T)$ gives an alternating sum with factorials and an exact algorithm with $O(N^2)$ integer operations (Section~\ref{ssec:exact}).
The alternating sum is hard to estimate, since its terms are much larger than its value; we write it exactly as a moment of a positive measure, a continuous dual Hahn weight, which performs all cancellations at once (Section~\ref{ssec:moment}). The density of this measure lies between two functions of the form polynomial${}\times e^{-\pi x}$, and the constant $2/\pi$ comes from this $e^{-\pi x}$ (Section~\ref{ssec:asym}).
Everything up to the moment representation is combinatorial; the analysis uses one integral of Wilson (Fact~\ref{fact:W}) and standard identities for the Gamma function.

\subsection{The operator}\label{ssec:operator}
For functions $F$ on $\{0,1,2,\dots\}$ with $F(0)=0$ let
\[
  (TF)(x):=x^2\bigl(F(x+1)+2F(x)+F(x-1)\bigr)\qquad(x\ge1),\qquad (TF)(0):=0 .
\]
\begin{lemma}[Operator form]\label{lem:operator}
For $q\ge1$, $\ W(2q+1)=4\,(T^{q-1}x^3)(1)$ and $\ W(2q+2)=2\,(T^qx^2)(1)$.
\end{lemma}
\begin{proof}
The transfer operator of the paths of Theorem~\ref{thm:motzkin} is $(Jf)(x)=x(x+1)f(x+1)+2x^2f(x)+x(x-1)f(x-1)$: The total weight of $q$-step paths from height $x$, with final weight $f$, is $(J^qf)(x)$. So $W(2q+1)=(J^q1)(1)$ and $W(2q+2)=(J^q\,2x)(1)$.
Expanding shows $x\,(Jf)(x)=T(xf)(x)$, hence $x\,(J^qf)(x)=(T^q(xf))(x)$, and at $x=1$ we get $W(2q+1)=(T^qx)(1)$ and $W(2q+2)=2(T^qx^2)(1)$. Finally $Tx=x^2\bigl((x+1)+2x+(x-1)\bigr)=4x^3$.
\end{proof}

\subsection{The bridge lemma}\label{ssec:bridge}
For the two parities let
\[
  \lambda^{\mathrm o}_j:=(2j+2)(2j+1)\quad(N\text{ odd}),\qquad\lambda^{\mathrm e}_j:=(2j+1)(2j)\quad(N\text{ even}),\qquad p_k(\omega):=\prod_{j<k}(\omega+\lambda_j),
\]
with $p_0=1$ and $\lambda$ taken from the relevant parity; $p_k(T)=\prod_{j<k}(T+\lambda_j)$ acts on functions.

\begin{lemma}[Bridge lemma]\label{thm:bridge}
For every $F$ with $F(0)=0$ and every $k\ge0$,
\[
  \bigl(p_k(T)F\bigr)(1)=(k!)^2\sum_{y=1}^{k+1}\beta_k(y)\,F(y),\qquad
  \beta_k(y)=\begin{cases}\binom{2k+2}{k+1-y},&N\text{ odd},\\[2pt] \frac{y}{k+1}\binom{2k+2}{k+1-y},&N\text{ even}.\end{cases}
\]
\end{lemma}
\noindent The coefficients are positive: The number $\binom{2k+2}{k+1-y}$ counts $\pm1$ walks of length $2k+2$ from $0$ to $2y$, and $\frac y{k+1}\binom{2k+2}{k+1-y}$ counts those that do not return to $0$ (a ballot number).
\begin{proof}
Induction on $k$; for $k=0$ both sides are $F(1)$. For a linear functional $\mathcal L(F)=\sum_{y\ge1}a(y)F(y)$ (with $a(0)=0$), the definition of $T$ and $F(0)=0$ give
\[
  \mathcal L(TF)=\sum_{y\ge1}F(y)\bigl[(y-1)^2a(y-1)+2y^2a(y)+(y+1)^2a(y+1)\bigr].
\]
So it suffices to verify, with $n=k+1$, the recurrence $n^2\beta_{n}(y)=(y-1)^2\beta_k(y-1)+\bigl(2y^2+\lambda_k\bigr)\beta_k(y)+(y+1)^2\beta_k(y+1)$ (the factor $n^2$ comes from $(k!)^2\to(n!)^2$).
For odd $N$, $\lambda_k=2n(2n-1)$, and we need
\[
  n^2\tbinom{2n+2}{n+1-y}=(y-1)^2\tbinom{2n}{n+1-y}+\bigl(2y^2+2n(2n-1)\bigr)\tbinom{2n}{n-y}+(y+1)^2\tbinom{2n}{n-1-y}.
\]
For $1\le y\le n$, divide by $\binom{2n}{n-y}$ and multiply by $(n+1)^2-y^2$. The left side becomes $2n^2(n+1)(2n+1)$. On the right, the first and third terms give $2y^4+(2n^2-6n-2)y^2+2n^2+2n$ and the middle term gives $-2y^4-(2n^2-6n-2)y^2+4n^4+6n^3-2n$; the sum is $4n^4+6n^3+2n^2=2n^2(n+1)(2n+1)$.
For $y=n+1$ only the first term survives and both sides equal $n^2$; for $y>n+1$ all terms vanish.
For even $N$, $\lambda_k=(2n-1)(2n-2)$, and the same computation with the extra factor $y/n$ gives $2n^3(2n+1)y$ on both sides after multiplication by $n((n+1)^2-y^2)$ for $1\le y\le n$; for $y=n+1$ both sides equal $n^2$ again.
\end{proof}

\begin{corollary}\label{cor:bridge}
For odd $N$, $(p_k(T)x^3)(1)=\frac{(2k+2)!}2$; for even $N$, $(p_k(T)x^2)(1)=(2k+1)!$.
\end{corollary}
\begin{proof}
Let $n=k+1$. The ballot identity $\frac yn\binom{2n}{n-y}=\binom{2n-1}{n-y}-\binom{2n-1}{n-y-1}$ telescopes to $\sum_{y\ge1}y\binom{2n}{n-y}=\frac n2\binom{2n}n$. With $y^2=n^2-(n-y)(n+y)$ and $(n-y)(n+y)\binom{2n}{n-y}=2n(2n-1)\binom{2n-2}{n-1-y}$,
\[
  \sum_{y\ge1}y^3\tbinom{2n}{n-y}=n^2\cdot\tfrac n2\tbinom{2n}n-2n(2n-1)\cdot\tfrac{n-1}2\tbinom{2n-2}{n-1}=\tfrac{n^2}2\tbinom{2n}n .
\]
By Lemma~\ref{thm:bridge}, the odd case gives $(k!)^2\cdot\frac{n^2}2\binom{2n}n=\frac{(2k+2)!}2$, and the even case gives $(k!)^2\cdot\frac1n\cdot\frac{n^2}2\binom{2n}n=(2k+1)!$.
\end{proof}

\subsection{An alternating sum and a triangle}\label{ssec:exact}
\begin{lemma}[Triangle]\label{lem:trec}
Let $t_{0,0}=1$, $t_{s,k}=0$ outside $0\le k\le s$, and $t_{s,k}=t_{s-1,k-1}+\lambda_k\,t_{s-1,k}$. Then, for each parity, $\omega^{s}=\sum_{k=0}^{s}(-1)^{s-k}\,t_{s,k}\,p_k(\omega)$ as polynomials in $\omega$.
\end{lemma}
\begin{proof}
Induction on $s$, using $\omega p_k=p_{k+1}-\lambda_kp_k$:
$\omega\cdot\omega^{s-1}=\sum_k(-1)^{s-1-k}t_{s-1,k}(p_{k+1}-\lambda_kp_k)=\sum_k(-1)^{s-k}\bigl[t_{s-1,k-1}+\lambda_kt_{s-1,k}\bigr]p_k$.
\end{proof}
\noindent The recurrence gives $t_{s,k}=h_{s-k}(\lambda_0,\dots,\lambda_k)$, a complete homogeneous symmetric function. In terms of the Jacobi--Stirling numbers of the second kind $\mathrm{JS}_n^k(z)$ of Gelineau and Zeng \cite{gelineau} ($\mathrm{JS}_n^k=\mathrm{JS}_{n-1}^{k-1}+k(k+z)\mathrm{JS}_{n-1}^k$),
\[
  t^{\mathrm e}_{s,k}=4^{s-k}\,\mathrm{JS}_s^k(\tfrac12),\qquad t^{\mathrm o}_{s,k}=4^{s-k}\,\mathrm{JS}_{s+1}^{k+1}(-\tfrac12),
\]
since both sides are $h_{s-k}$ at the same nodes ($\lambda^{\mathrm e}_j=4j(j+\frac12)$ and $\lambda^{\mathrm o}_j=4i(i-\frac12)$ with $i=j+1$; a zero node does not change $h$). The two sets of nodes, $K(K+1)$ for $K$ even and for $K$ odd, together form the Legendre eigenvalues $K(K+1)$; since $h$ of a union of two disjoint sets of nodes is $\sum_ih_i(\text{first set})\,h_{n-i}(\text{second set})$, the Legendre--Stirling numbers of the second kind \cite{andrews}, $\mathrm{PS}(n,k)=h_{n-k}(0\cdot1,1\cdot2,\dots,k(k+1))$, satisfy $\mathrm{PS}(n,k)=\sum_i t^{\mathrm e}_{a+i,a}\,t^{\mathrm o}_{b+n-k-i,b}$ with $a=\lfloor k/2\rfloor$, $b=\lceil k/2\rceil-1$ ($k\ge1$).

\begin{theorem}[Alternating sum]\label{thm:altsum}
For $q\ge1$,
\[
  W(2q+1)=2\sum_{k=0}^{q-1}(-1)^{q-1-k}\,t^{\mathrm o}_{q-1,k}\,(2k+2)!,\qquad
  W(2q+2)=2\sum_{k=0}^{q}(-1)^{q-k}\,t^{\mathrm e}_{q,k}\,(2k+1)!.
\]
\end{theorem}
\begin{proof}
Lemma~\ref{lem:trec} is a polynomial identity, so $\omega$ may be replaced by $T$. Apply it to $x^3$ (odd $N$) or $x^2$ (even $N$), evaluate at $x=1$, and use Lemma~\ref{lem:operator} and Corollary~\ref{cor:bridge}.
\end{proof}
\noindent For $F(x)=x^3$ (odd case) or $F(x)=x^2$ (even case), $(p_k(T)F)(1)$ is the total weight of ``lazy'' paths, which at each step either take a step of $T$ or pause with weight $\lambda_j$, with terminal weight $F(y)$ at the final height $y$; the alternating sum removes the pauses by inclusion--exclusion. Appendix~\ref{sec:involution} gives a second proof, directly from $W(N)=\per M$ by a sign-reversing involution.
For $N=7$ (odd, $q=3$) we have $\lambda^{\mathrm o}_0=2$, $\lambda^{\mathrm o}_1=12$, $t_{2,\cdot}=(4,14,1)$, and $W(7)=2\,(4\cdot2!-14\cdot4!+1\cdot6!)=2\,(8-336+720)=784$. The cancellation grows with $N$: For $N=20$ the term of largest absolute value is about $368.5$ times $W(20)/2$.
Theorem~\ref{thm:altsum} computes $W(N)$ with $O(N^2)$ integer operations; the values in Tables~\ref{tab:values} and~\ref{tab:asym} were computed in this way.

\subsection{Moment representation}\label{ssec:moment}
\begin{fact}[Total mass of the continuous dual Hahn weight; Wilson \cite{wilson}, {\cite[(18.25.2), (18.25.7)--(18.25.8)]{dlmf}}]\label{fact:W}
For $\alpha,\beta,\gamma>0$, $\ \displaystyle\frac1{2\pi}\int_0^\infty\Bigl|\frac{\Gamma(\alpha+ix)\Gamma(\beta+ix)\Gamma(\gamma+ix)}{\Gamma(2ix)}\Bigr|^2dx=\Gamma(\alpha+\beta)\Gamma(\alpha+\gamma)\Gamma(\beta+\gamma)$.
\end{fact}
\noindent Let $\mu_{\alpha,\beta,\gamma}$ be the probability measure on $(0,\infty)$ whose density is the integrand above divided by $2\pi\,\Gamma(\alpha+\beta)\Gamma(\alpha+\gamma)\Gamma(\beta+\gamma)$, and let $(z)_k:=z(z+1)\ldots(z+k-1)$.
\begin{lemma}\label{lem:cdh}
$\Ex_{\mu_{\alpha,\beta,\gamma}}\bigl[(\alpha+ix)_k(\alpha-ix)_k\bigr]=(\alpha+\beta)_k(\alpha+\gamma)_k$; in particular all moments are finite.
\end{lemma}
\begin{proof}
$|\Gamma(\alpha+ix)|^2(\alpha+ix)_k(\alpha-ix)_k=|\Gamma(\alpha+k+ix)|^2$, so the left side is the total mass with $\alpha$ replaced by $\alpha+k$ divided by the original total mass, which by Fact~\ref{fact:W} is the right side. Moreover $(\alpha+ix)_k(\alpha-ix)_k=\prod_{j<k}((\alpha+j)^2+x^2)\ge x^{2k}$.
\end{proof}
\begin{theorem}[Moment representation]\label{thm:moment}
Let $\omega:=4\bigl(x^2+\frac1{16}\bigr)=4\bigl|x+\frac i4\bigr|^2$, $\mu_{\mathrm o}:=\mu_{3/4,3/4,5/4}$ and $\mu_{\mathrm e}:=\mu_{1/4,3/4,5/4}$. For every polynomial $Q$,
\[
  \bigl(Q(T)x^3\bigr)(1)=\Ex_{\mu_{\mathrm o}}\bigl[Q(\omega)\bigr],\qquad \bigl(Q(T)x^2\bigr)(1)=\Ex_{\mu_{\mathrm e}}\bigl[Q(\omega)\bigr].
\]
Consequently $W(2q+1)=4\,\Ex_{\mu_{\mathrm o}}[\omega^{q-1}]$ and $W(2q+2)=2\,\Ex_{\mu_{\mathrm e}}[\omega^{q}]$.
\end{theorem}
\begin{proof}
Both sides are linear in $Q$, so it suffices to take $Q=p_k$ (with $\lambda=\lambda^{\mathrm o}$ in the first identity and $\lambda=\lambda^{\mathrm e}$ in the second); the left side is given by Corollary~\ref{cor:bridge}. For the right side in the odd case, $\frac14+(2j+2)(2j+1)=4(j+\frac34)^2$, so $\omega+\lambda^{\mathrm o}_j=4(\frac34+j+ix)(\frac34+j-ix)$ and $p_k(\omega)=4^k(\frac34+ix)_k(\frac34-ix)_k$.
Lemma~\ref{lem:cdh} with $(\frac34,\frac34,\frac54)$ gives $\Ex_{\mu_{\mathrm o}}p_k=4^k(\frac32)_k(2)_k=\frac{(2k+2)!}2$.
In the even case $\frac14+(2j+1)(2j)=4(j+\frac14)^2$, $p_k(\omega)=4^k(\frac14+ix)_k(\frac14-ix)_k$, and $(\frac14,\frac34,\frac54)$ gives $\Ex_{\mu_{\mathrm e}}p_k=4^k(1)_k(\frac32)_k=(2k+1)!$. Finally use Lemma~\ref{lem:operator}.
\end{proof}
\noindent For $N=7$, Lemma~\ref{lem:trec} gives $\omega^2=p_2-14p_1+4p_0$, so $\Ex_{\mu_{\mathrm o}}\omega^2=360-168+4=196$ and $W(7)=4\cdot196=784$; the number $196$ is the integral of a positive function.

\textbf{A unified formula.} In the odd case one factor $\omega$ can be absorbed: The densities of $\mu_{\mathrm o}$ and $\mu^\ast:=\mu_{1/4,3/4,3/4}$ have ratio proportional to $|\Gamma(\frac54+ix)|^2/|\Gamma(\frac14+ix)|^2=x^2+\frac1{16}=\omega/4$, and $\Ex_{\mu^\ast}[\omega/4]=1$ by Lemma~\ref{lem:cdh} with $k=1$; so $d\mu_{\mathrm o}=\frac\omega4\,d\mu^\ast$ and
\[
  W(2s+1)=\Ex_{\mu_{1/4,3/4,3/4}}[\omega^s],\qquad W(2s+2)=2\,\Ex_{\mu_{1/4,3/4,5/4}}[\omega^s]\qquad(s\ge0)
\]
(for $s=0$ these read $W(1)=1$ and $W(2)=2$).
Both parities use $\mu_{1/4,3/4,\gamma}$, with $\gamma=\frac34$ or $\frac54$. The parameters $\frac14$ and $\frac34$ are related by the reflection formula $\Gamma(z)\Gamma(1-z)=\pi/\sin\pi z$; with the duplication formula it gives $|\Gamma(\frac14+ix)\Gamma(\frac34+ix)|^2/|\Gamma(2ix)|^2=4\pi x\tanh(2\pi x)$ (see the proof of Lemma~\ref{lem:dens}), so $\mu_{1/4,3/4,\gamma}$ has density $\frac4{\sqrt\pi}\,x\tanh(2\pi x)\,|\Gamma(\gamma+ix)|^2$. Substituting $|\Gamma(\frac54+ix)|^2=(x^2+\frac1{16})|\Gamma(\frac14+ix)|^2$ combines the two parities:
\begin{equation}\label{eq:unified}
  W(N)=\frac{2^{N+1}}{\sqrt\pi}\int_0^\infty x\tanh(2\pi x)\Bigl(x^2+\frac1{16}\Bigr)^{\lfloor N/2\rfloor}\Bigl|\Gamma\Bigl(\frac14+\Bigl\{\frac N2\Bigr\}+ix\Bigr)\Bigr|^2dx\qquad(N\ge1),
\end{equation}
where $\{\cdot\}$ denotes the fractional part. The reflection formula is used only to turn the common pair $(\frac14,\frac34)$ into the elementary factor $x\tanh(2\pi x)$; the parity sits entirely in the third parameter.

\subsection{The Genocchi family}\label{ssec:genocchi}
By the theory of Flajolet \cite{flajolet} and Viennot \cite{viennot}, the weights of Motzkin paths are the coefficients of a J-fraction. In the variable $u=x^2+\alpha^2$, the monic orthogonal polynomials $P_n$ of $\mu_{\alpha,\beta,\gamma}$ (the continuous dual Hahn polynomials \cite[\S9.3]{kls}) satisfy
$u\,P_n=P_{n+1}+(A_n+C_n)\,P_n+A_{n-1}C_n\,P_{n-1}$ ($P_{-1}=0$, $P_0=1$) with $A_n=(n+\alpha+\beta)(n+\alpha+\gamma)$ and $C_n=n(n+\beta+\gamma-1)$.
This J-fraction is the even contraction of the S-fraction with $\theta_{2k-1}=A_{k-1}$ and $\theta_{2k}=C_k$ (with $\theta_0:=0$; indeed $C_0=0$, $A_n+C_n=\theta_{2n}+\theta_{2n+1}$ and $A_{n-1}C_n=\theta_{2n-1}\theta_{2n}$). The definition of $\mu_{\alpha,\beta,\gamma}$ extends to $\alpha=0$, and Fact~\ref{fact:W} and Lemma~\ref{lem:cdh} remain valid by dominated convergence as $\alpha\to0^+$: For $0<\alpha\le1$, Fact~\ref{fact:P} gives $|\Gamma(\alpha+ix)|^2=|\Gamma(1+\alpha+ix)|^2/(\alpha^2+x^2)\le\Gamma(1+\alpha)^2|\Gamma(2+ix)|^2/(\alpha^2+x^2)$. With $|\Gamma(2ix)|^{-2}=2x\sinh(2\pi x)/\pi$, the integrand of Fact~\ref{fact:W}, times any fixed polynomial in $x^2$, is therefore bounded on $(0,1]$ uniformly in $\alpha$, and on $[1,\infty)$ it is at most a constant times the same integrand with $\alpha$ replaced by $2$, which is integrable; for example $d\mu_{0,1,1}=2\pi x^2\cosh(\pi x)\sinh^{-2}(\pi x)\,dx$. For the four measures below, $\theta_{2k-1}=k(k+a')$ and $\theta_{2k}=k(k+b')$:
\begin{center}\small
\begin{tabular}{@{}lll@{}}
\toprule
$(a',b')$ & measure & moments\\
\midrule
$(0,1)$ & $\mu_{0,1,1}$ & Genocchi numbers (A110501)\\
$(1,1)$ & $\mu_{1/2,1/2,3/2}$ & median Genocchi numbers $H_{2s+3}$ (A005439)\\
$(0,\frac12)$ & $\mu_{1/4,3/4,3/4}$ & $W(2s+1)/4^s$\\
$(\frac12,1)$ & $\mu_{1/4,3/4,5/4}$ & $W(2s+2)/(2\cdot4^s)$\\
\bottomrule
\end{tabular}
\end{center}
In the variable $\omega$ itself the S-fractions of $W$ have integer coefficients. The number $W(2s+1)$ is the total weight of the Dyck paths of length $2s$ in which a down step from height $h$ has weight $\theta_h$, where $\theta_{2k-1}=(2k)^2$ and $\theta_{2k}=2k(2k+1)$ ($4,6,16,20,36,42,\dots$); for $W(2s+2)/2$ the weights are $2k(2k+1)$ and $2k(2k+2)$ ($6,8,20,24,\dots$).
The median Genocchi numbers also occur directly among rainbow maps:
\begin{corollary}\label{cor:genocchi}
The number of rainbow maps of order $N=2q+1$ with exactly $q$ arcs pointing to the nearer end is the median Genocchi number $H_{2q+3}$ ($2,8,56,608,9440,\dots$ for $q=1,2,\dots$), the number of D-permutations of $[2q]$ \cite{debsokal}, that is, of permutations $\pi$ of $\{1,\dots,2q\}$ with $\pi(i)\ge i$ for odd $i$ and $\pi(i)\le i$ for even $i$.
\end{corollary}
\begin{proof}
Exactly $q$ arcs point to the nearer end if and only if $S_q=q$, that is, $x_{q+1}=1$. So the count is the total weight of the paths of Section~\ref{ssec:pairs} from height $1$ back to height $1$. With height $y=x-1$, this is the $q$th moment of the J-fraction with level weights $2(y+1)^2$ and products of up and down weights $y^2(y+1)^2$, which is the even contraction of the S-fraction $\theta_{2k-1}=\theta_{2k}=k(k+1)$ of the shifted median Genocchi numbers \cite[(2.38)--(2.39)]{debsokal}.
\end{proof}
\noindent We do not know a bijection between these rainbow maps and D-permutations, nor models of the same kind for the integer sequences $W(2s+1)$ and $W(2s+2)/2$.

\subsection{Asymptotics}\label{ssec:asym}
It remains to estimate the moments $\Ex[\omega^{s}]=\int4^s(x^2+\frac1{16})^s\,d\mu$, where $\mu$ is $\mu_{\mathrm o}$ or $\mu_{\mathrm e}$, up to a polynomial factor in $s$. For this it suffices to bound the density above and below by functions of the form polynomial${}\times e^{-\pi x}$: The $s$th moment of $x^2$ is then $\Gamma(2s+O(1))\,\pi^{-2s}$, and multiplying by $4^s$ gives $(2/\pi)^{2s}(2s)!$ up to a polynomial factor. A finer asymptotic formula is given in Appendix~\ref{sec:fine}; it is not needed for the main theorem.

\begin{lemma}[Density]\label{lem:dens}
Let $\alpha=\frac34$ (for $\mu_{\mathrm o}$) or $\frac14$ (for $\mu_{\mathrm e}$). The density of the measure is
$\ \psi_\alpha(x)=\pi^{-3/2}\cdot4\pi x\bigl(x^2+\frac1{16}\bigr)\tanh(2\pi x)\,|\Gamma(\alpha+ix)|^2$ for $x>0$.
\end{lemma}
\begin{proof}
We use three standard identities {\cite[(5.4.3), (5.4.4), (5.5.5)]{dlmf}}: $|\Gamma(iy)|^2=\frac{\pi}{y\sinh\pi y}$, $|\Gamma(\frac12+iy)|^2=\frac\pi{\cosh\pi y}$, and $\Gamma(2z)=\pi^{-1/2}2^{2z-1}\Gamma(z)\Gamma(z+\frac12)$.
The first gives $|\Gamma(2ix)|^{-2}=\frac{2x\sinh2\pi x}\pi$. The third with $z=\frac14+ix$ (so that $|2^{2z-1}|^2=\frac12$), together with the second, gives $|\Gamma(\frac14+ix)\Gamma(\frac34+ix)|^2=2\pi|\Gamma(\frac12+2ix)|^2=\frac{2\pi^2}{\cosh2\pi x}$. Also $|\Gamma(\frac54+ix)|^2=(x^2+\frac1{16})|\Gamma(\frac14+ix)|^2$.
In both cases the product of the three squared Gamma moduli is therefore $\frac{2\pi^2}{\cosh2\pi x}(x^2+\frac1{16})|\Gamma(\alpha+ix)|^2$; for $(\frac34,\frac34,\frac54)$, pair one factor $|\Gamma(\frac34+ix)|^2$ with $|\Gamma(\frac54+ix)|^2=(x^2+\frac1{16})|\Gamma(\frac14+ix)|^2$.
Multiplying by $|\Gamma(2ix)|^{-2}$ gives $4\pi x(x^2+\frac1{16})\tanh(2\pi x)|\Gamma(\alpha+ix)|^2$, and the normalization $2\pi$ times the right side of Fact~\ref{fact:W} is $2\pi\,\Gamma(\frac32)\Gamma(2)\Gamma(2)=\pi^{3/2}$ for $\mu_{\mathrm o}$ and $2\pi\,\Gamma(1)\Gamma(\frac32)\Gamma(2)=\pi^{3/2}$ for $\mu_{\mathrm e}$.
\end{proof}

\begin{fact}[{\cite[(5.8.3)]{dlmf}}]\label{fact:P}
For real $a>0$ and real $x$, $\ \displaystyle\frac{\Gamma(a)^2}{|\Gamma(a+ix)|^2}=\prod_{k=0}^{\infty}\Bigl(1+\frac{x^2}{(a+k)^2}\Bigr)$.
In particular, for fixed $x$, $|\Gamma(a+ix)|^2/\Gamma(a)^2$ is nondecreasing in $a$.
\end{fact}

\begin{lemma}[Elementary bounds for the density]\label{lem:sandwich}
In both cases, $\ \psi_\alpha(x)\le60\,(x+1)^4e^{-\pi x}$ for $x>0$ and $\ \psi_\alpha(x)\ge3\,x^2e^{-\pi x}$ for $x\ge1$.
\end{lemma}
\begin{proof}
We have $|\Gamma(\frac12+ix)|^2=\frac\pi{\cosh\pi x}$ and, from $\Gamma(1+ix)=ix\,\Gamma(ix)$, $|\Gamma(1+ix)|^2=\frac{\pi x}{\sinh\pi x}$. Moreover, for $x>0$,
\[
  e^{-\pi x}\le\frac1{\cosh\pi x}\le2e^{-\pi x},\qquad 2xe^{-\pi x}\le\frac{x}{\sinh\pi x}\le2(x+1)e^{-\pi x};
\]
the last inequality is equivalent to $(x+1)e^{-2\pi x}\le1$.

\emph{Upper bound.} By the monotonicity in Fact~\ref{fact:P} (comparing $a=\frac34$ with $1$, and $a=\frac14$ with $\frac12$),
\begin{align*}
  |\Gamma(\tfrac34+ix)|^2&\le\Gamma(\tfrac34)^2\,|\Gamma(1+ix)|^2\le2\pi\,\Gamma(\tfrac34)^2(x+1)e^{-\pi x},\\
  |\Gamma(\tfrac14+ix)|^2&\le\frac{\Gamma(\frac14)^2}{\Gamma(\frac12)^2}\,|\Gamma(\tfrac12+ix)|^2\le2\,\Gamma(\tfrac14)^2e^{-\pi x}.
\end{align*}
With $\tanh\le1$ and $x(x^2+\frac1{16})\le(x+1)^3$, Lemma~\ref{lem:dens} gives $\psi_{3/4}\le8\sqrt\pi\,\Gamma(\frac34)^2(x+1)^4e^{-\pi x}$ and $\psi_{1/4}\le\frac8{\sqrt\pi}\Gamma(\frac14)^2(x+1)^3e^{-\pi x}$. Since $\Gamma(\frac14)=3.62560\ldots$ and $\Gamma(\frac34)=1.22541\ldots$, we have $8\sqrt\pi\,\Gamma(\frac34)^2<21.3$ and $\frac8{\sqrt\pi}\Gamma(\frac14)^2<59.34$.

\emph{Lower bound ($x\ge1$).} Again by monotonicity (comparing $a=\frac34$ with $\frac12$, and $a=\frac54$ with $1$),
\begin{align*}
  |\Gamma(\tfrac34+ix)|^2&\ge\frac{\Gamma(\frac34)^2}{\Gamma(\frac12)^2}\,|\Gamma(\tfrac12+ix)|^2\ge\Gamma(\tfrac34)^2e^{-\pi x},\\
  |\Gamma(\tfrac14+ix)|^2&=\frac{|\Gamma(\frac54+ix)|^2}{x^2+\frac1{16}}\ge\frac{\Gamma(\frac54)^2\,|\Gamma(1+ix)|^2}{x^2+\frac1{16}}\ge\frac{2\pi\,\Gamma(\frac54)^2\,x\,e^{-\pi x}}{x^2+\frac1{16}} .
\end{align*}
For $x\ge1$, $\tanh2\pi x>0.99$ and $x(x^2+\frac1{16})\ge x^2$, so $\psi_{3/4}\ge0.99\cdot\frac4{\sqrt\pi}\Gamma(\frac34)^2\,x^2e^{-\pi x}$ and $\psi_{1/4}\ge0.99\cdot8\sqrt\pi\,\Gamma(\frac54)^2\,x^2e^{-\pi x}$; the two coefficients exceed $3.35$ and $11.5$.
\end{proof}

\begin{proof}[Proof of Theorem~\ref{thm:perm}]
Since $\log n!=\sum_{k=2}^n\log k$ lies between $\int_1^n\log x\,dx$ and $\int_1^n\log x\,dx+\log n$,
\begin{equation}\label{eq:fact}
  n\log n-n+1\ \le\ \log n!\ \le\ n\log n-n+1+\log n\qquad(n\ge1).
\end{equation}
We also use $\int_0^\infty x^{n}e^{-\pi x}dx=n!\,\pi^{-(n+1)}$.

\emph{Upper bound for the moments.} By $x^2+\frac1{16}\le(x+1)^2$ and Lemma~\ref{lem:sandwich}, for integers $s\ge0$,
\[
  \Ex[\omega^s]\le60\cdot4^s\int_0^\infty(x+1)^{2s+4}e^{-\pi x}dx\le60e^{\pi}\,4^s\int_0^\infty u^{2s+4}e^{-\pi u}du=60e^{\pi}\,4^s\,\frac{(2s+4)!}{\pi^{2s+5}} .
\]
\emph{Lower bound for the moments.} By $x^2+\frac1{16}\ge x^2$ and Lemma~\ref{lem:sandwich},
\[
  \Ex[\omega^s]\ge3\cdot4^s\int_1^\infty x^{2s+2}e^{-\pi x}dx\ge3\cdot4^s\Bigl(\frac{(2s+2)!}{\pi^{2s+3}}-1\Bigr)\ge\frac32\cdot4^s\,\frac{(2s+2)!}{\pi^{2s+3}}\qquad(s\ge4),
\]
where the last step uses $(2s+2)!\ge2\pi^{2s+3}$: For $s=4$, $10!/\pi^{11}>12.3$, and increasing $s$ by one multiplies the ratio by $(2s+3)(2s+4)/\pi^2>1$.

\emph{Substituting.} For $m=2s+2$ ($N$ odd, $W(N)=4\Ex_{\mu_{\mathrm o}}[\omega^s]$ by Theorem~\ref{thm:moment}) and for $m=2s+1$ ($N$ even, $W(N)=2\Ex_{\mu_{\mathrm e}}[\omega^s]$), using $4^s\pi^{-2s}=(\frac2\pi)^m\frac{\pi^2}4$ and $(\frac2\pi)^m\frac\pi2$ respectively,
\begin{align*}
  \frac{3}{2\pi}\Bigl(\frac2\pi\Bigr)^{m}m!\ &\le\ W(N)\ \le\ \frac{60e^\pi}{\pi^3}\Bigl(\frac2\pi\Bigr)^{m}(m+2)!&&(N\text{ odd},\ m\ge10),\\
  \frac{3}{2\pi^2}\Bigl(\frac2\pi\Bigr)^{m}(m+1)!\ &\le\ W(N)\ \le\ \frac{60e^\pi}{\pi^4}\Bigl(\frac2\pi\Bigr)^{m}(m+3)!&&(N\text{ even},\ m\ge9).
\end{align*}
Since $\frac3{2\pi}>0.477$, $\frac{60e^\pi}{\pi^3}<44.8$, $\frac3{2\pi^2}>0.151$ and $\frac{60e^\pi}{\pi^4}<14.3$, the first assertion holds for $m\ge9$. For $N=3,\dots,9$ ($m\le8$), Table~\ref{tab:values} gives $W(N)\big/\bigl((\frac2\pi)^mm!\bigr)\approx4.93$, $7.75$, $10.15$, $13.39$, $16.36$, $20.00$, $23.44$, all between $\frac17$ and $45$.

\emph{Second assertion.} Let $\Theta_N:=N\log\frac{2N}\pi-N$, and note $(\frac2\pi)^m=\frac\pi2(\frac2\pi)^N$, $m!=N!/N$ and $(m+3)!=N!\,(N+1)(N+2)$. By \eqref{eq:fact},
\begin{align*}
  \log W(N)&\ge-\log7+\log\tfrac\pi2+N\log\tfrac2\pi+(N\log N-N+1)-\log N\\
  &\ge\Theta_N-\log N-0.50\ \ge\ \Theta_N-\Lambda,\\
  \log W(N)&\le\log45+\log\tfrac\pi2+N\log\tfrac2\pi+(N\log N-N+1+\log N)+2\log(N+2)\\
  &\le\Theta_N+3\log(N+2)+5.26 .
\end{align*}
For $N\ge3$, $\log(N+2)\le\log(N+1)+\log\frac54\le\Lambda-0.77$ and $\Lambda\ge2.38$, so $\log W(N)\le\Theta_N+3\Lambda+2.95\le\Theta_N+5\Lambda$.
\end{proof}

%======================================================================
\section{Concluding remarks}\label{sec:concl}
We have proved $\log\tau(N)=N\log N-N\log\frac eb+O(\log^4N)$ with $b=2/\pi$, with the explicit error bound $152\Lambda^4$ for every $N\ge3$. The comparison of $\tau(N)$ with $W(N)$ (Sections~\ref{sec:cycles}--\ref{sec:count}) uses only an injection, a count of vertex-disjoint cycles, an inequality on words, and an exchange argument; the formulas for $W(N)$ start from the Motzkin paths of Section~\ref{sec:struct}.

Some questions remain open.
(1) Anick's ratio form $\tau(N)/\tau(N-1)=a+bN+(-1)^Nc/N+\cdots$ and the constants $a$, $c$. If the ratio form holds with remainder $O(N^{-2})$, summing logarithms gives $\log\tau(N)=N\log N-N\log\frac eb+(\frac ab+\frac12)\log N+O(1)$; the ratio form alone does not determine $b$, but together with Theorem~\ref{thm:main} it forces $b=2/\pi$.
(2) An improvement of the error $O(\log^4N)$ to $O(\log N)$, and the polynomial order of $\tau(N)/(N!\,(2/\pi)^N)$.
(3) The distribution of the number of cycles of a uniformly random rainbow map, and a combinatorial interpretation of the acyclic probability $\tau(N)/W(N)$ (which is $0.7297$ for $N=30$).

\section*{Acknowledgments}
I thank Donald E.~Knuth for his detailed comments on an earlier version of this paper.
This work was done with extensive use of AI systems directed by the author; the author is responsible for all of its content.

\appendix
\section{A finer asymptotic formula}\label{sec:fine}
Theorem~\ref{thm:perm} determines $\log W(N)$ up to $O(\log N)$. The moment representation gives the first terms of the full asymptotic expansion.
\begin{theorem}\label{thm:fine}
As $N\to\infty$, $\ W(N)=G_N\bigl(1+\frac{\pi^2}{32N}+O(N^{-2})\bigr)$ with $\ G_N:=\frac4\pi\bigl(\frac2\pi\bigr)^{N}\Gamma\bigl(N+\tfrac32\bigr)$; equivalently
\[
  \log W(N)=N\log\frac{2N}\pi-N+\log N+\log\frac{4\sqrt{2\pi}}\pi+\frac{\varkappa}N+O(N^{-2}),\qquad\varkappa=\frac{11}{24}+\frac{\pi^2}{32}.
\]
\end{theorem}
\begin{fact}[Stirling's formula; {\cite[(5.11.1)]{dlmf}}]\label{fact:S}
As $z\to\infty$ with $|\arg z|\le\pi-\delta$, $\ \log\Gamma(z)=(z-\frac12)\log z-z+\frac12\log2\pi+\frac1{12z}+O(|z|^{-3})$.
\end{fact}
\begin{lemma}\label{lem:densasym}
As $x\to\infty$, $\ \psi_\alpha(x)=8\sqrt\pi\,x^{2\alpha+2}e^{-\pi x}\bigl(1+O(x^{-2})\bigr)$.
\end{lemma}
\begin{proof}
Take $z=\alpha+ix$ in Fact~\ref{fact:S} and take real parts. With $\theta:=\arctan\frac\alpha x$ we have $\arg z=\frac\pi2-\theta$, $\log|z|=\log x+O(x^{-2})$, $x\theta=\alpha+O(x^{-2})$ and $\mathrm{Re}\frac1{12z}=O(x^{-2})$, so
\begin{align*}
  \log|\Gamma(\alpha+ix)|&=\bigl(\alpha-\tfrac12\bigr)\log|z|-x\bigl(\tfrac\pi2-\theta\bigr)-\alpha+\tfrac12\log2\pi+O(x^{-2})\\
  &=\bigl(\alpha-\tfrac12\bigr)\log x-\tfrac{\pi x}2+\tfrac12\log2\pi+O(x^{-2}).
\end{align*}
Hence $|\Gamma(\alpha+ix)|^2=2\pi x^{2\alpha-1}e^{-\pi x}(1+O(x^{-2}))$. Multiply by $x^2+\frac1{16}=x^2(1+O(x^{-2}))$ and $\tanh2\pi x=1+O(e^{-4\pi x})$, and note $\pi^{-3/2}\cdot4\pi\cdot2\pi=8\sqrt\pi$.
\end{proof}
\begin{lemma}[Asymptotics of moments]\label{lem:watson}
Let $\psi(x)=Cx^{\nu}e^{-\pi x}(1+\eta(x))$ be a probability density with $C,\nu>0$ and $|\eta(x)|\le K/x^2$ for $x\ge x_0$. Then, as $s\to\infty$,
\[
  \int_0^\infty\bigl(4(x^2+\tfrac1{16})\bigr)^{s}\psi(x)\,dx=C\,4^{s}\,\frac{\Gamma(2s+\nu+1)}{\pi^{2s+\nu+1}}\Bigl(1+\frac{\pi^2}{64s}+O(s^{-2})\Bigr).
\]
\end{lemma}
\begin{proof}
Let $n:=2s+\nu$ and $g_j:=\int_0^\infty x^{n-j}e^{-\pi x}dx=\Gamma(n-j+1)\,\pi^{-(n-j+1)}$, so that $\frac{g_2}{g_0}=\frac{\pi^2}{n(n-1)}$ and $\frac{g_4}{g_0}=\frac{\pi^4}{n(n-1)(n-2)(n-3)}$. Let $\xi:=s^{2/3}$, and let $s$ be large enough that $\xi\ge\max(1,x_0)$ and $\pi\xi\le n-4$.

\emph{(a) The model integrals have negligible mass on $[0,\xi]$.} For $0\le j\le4$, $x^{n-j}e^{-\pi x}$ increases on $[0,(n-j)/\pi]\supseteq[0,\xi]$, so $\int_0^\xi x^{n-j}e^{-\pi x}dx\le\xi^{n+1}$. Since $\Gamma(n+1)\ge\int_n^\infty x^ne^{-x}dx\ge(n/e)^n$, we have $g_0\ge\pi^{-(n+1)}(n/e)^n$ and
\[
  \frac{\xi^{n+1}}{g_0}\ \le\ \delta_s:=\pi\xi\Bigl(\frac{e\pi\xi}{n}\Bigr)^{n},\qquad\frac{e\pi\xi}{n}\le\frac{e\pi}2\,s^{-1/3}\to0 ,
\]
so $\delta_s$ is smaller than any negative power of $s$.

\emph{(b) The range $x\ge\xi$.} Let $y:=\frac1{16x^2}$; then $\bigl(4(x^2+\frac1{16})\bigr)^s=4^sx^{2s}(1+y)^s$ and $0\le sy\le\frac1{16}$. With $u:=s\log(1+y)$, the bounds $y-\frac{y^2}2\le\log(1+y)\le y$ give $sy-\frac{sy^2}2\le u\le sy\le1$, and $1+u\le e^u\le1+u+u^2$ gives
$(1+y)^s=1+sy+\mathcal R_1$ with $|\mathcal R_1|\le\tfrac12sy^2+s^2y^2\le2s^2y^2$.
Including the factor $1+\eta(x)$, for $x\ge\xi$,
\[
  \bigl(4(x^2+\tfrac1{16})\bigr)^{s}\psi(x)=C\,4^{s}x^{n}e^{-\pi x}\Bigl[1+\frac{s}{16x^2}+\mathcal R(x)\Bigr],\qquad
  |\mathcal R(x)|\le K'\Bigl(\frac{s^2}{x^4}+\frac1{x^2}\Bigr),
\]
with $K'$ depending only on $K$. Integrate over $[\xi,\infty)$; by (a), replacing $\int_\xi^\infty x^{n-j}e^{-\pi x}dx$ by $g_j$ ($j=0,2,4$) costs at most $g_0\delta_s$. Thus
\[
  \int_\xi^\infty\bigl(4(x^2+\tfrac1{16})\bigr)^{s}\psi\,dx
  =C4^{s}g_0\Bigl[1+\frac s{16}\frac{g_2}{g_0}+O\Bigl(s^2\frac{g_4}{g_0}+\frac{g_2}{g_0}+s^2\delta_s\Bigr)\Bigr]
  =C4^{s}g_0\Bigl[1+\frac{\pi^2}{64s}+O(s^{-2})\Bigr],
\]
because $\frac{s}{16}\cdot\frac{\pi^2}{n(n-1)}-\frac{\pi^2}{64s}=\frac{\pi^2}{16}\cdot\frac{4s^2-n(n-1)}{4s\,n(n-1)}=O(s^{-2})$, $s^2g_4/g_0=O(s^{-2})$ and $g_2/g_0=O(s^{-2})$.

\emph{(c) The range $x\le\xi$.} Since $\psi$ is a probability density,
$\int_0^\xi\bigl(4(x^2+\frac1{16})\bigr)^s\psi\,dx\le4^s\xi^{2s}\bigl(1+\frac1{16\xi^2}\bigr)^s\le e^{1/16}4^s\xi^{2s}\le e^{1/16}4^{s}g_0\delta_s$
(as $\xi\ge1$ and $2s\le n+1$), which is again negligible. Adding (b) and (c) proves the lemma.
\end{proof}
\begin{proof}[Proof of Theorem~\ref{thm:fine}]
By Lemma~\ref{lem:densasym}, Lemma~\ref{lem:watson} applies with $C=8\sqrt\pi$ and $\nu=2\alpha+2$. By Theorem~\ref{thm:moment}, for $m=2s+2$ (odd $N$) $W(N)=4\cdot8\sqrt\pi\,4^{s}\Gamma(2s+\frac92)\pi^{-(2s+9/2)}(1+\frac{\pi^2}{64s}+O(s^{-2}))$, and for $m=2s+1$ (even $N$) $W(N)=2\cdot8\sqrt\pi\,4^{s}\Gamma(2s+\frac72)\pi^{-(2s+7/2)}(1+\frac{\pi^2}{64s}+O(s^{-2}))$.
Both equal $\frac8{\pi^2}(\frac2\pi)^m\Gamma(m+\frac52)(1+\frac{\pi^2}{64s}+O(s^{-2}))$ with $s=\frac m2+O(1)$; substituting $m=N-1$ gives the first form.
For the second, Fact~\ref{fact:S} gives $\log\Gamma(N+\frac32)=(N+1)\log N-N+\frac12\log2\pi+\frac{11}{24N}+O(N^{-2})$, because $(N+1)\log(1+\frac3{2N})=\frac32+\frac3{8N}+O(N^{-2})$ and $\frac38+\frac1{12}=\frac{11}{24}$.
\end{proof}
\begin{table}[h]\centering\small
\caption{Numerical check of Theorem~\ref{thm:fine}. Here $\varepsilon_N:=\log W(N)-\bigl(N\log\frac{2N}\pi-N+\log N+\log\frac{4\sqrt{2\pi}}\pi\bigr)$; the theorem implies $N\varepsilon_N\to\varkappa=0.766758\ldots$ and $N(\log W(N)-\log G_N)\to\frac{\pi^2}{32}=0.308425\ldots$, and that the fourth column is bounded. (Computed with $50$ decimal digits.)}\label{tab:asym}
\begin{tabular}{@{}rrrrr@{}}
\toprule
$N$ & $\log W(N)$ & $N\varepsilon_N$ & $N^2(\varepsilon_N-\varkappa/N)$ & $N\bigl(\log W(N)-\log G_N\bigr)$\\
\midrule
16 & 25.117290 & 0.769631 & $+0.0460$ & 0.318942\\
17 & 27.524542 & 0.732785 & $-0.5776$ & 0.281655\\
64 & 178.598599 & 0.767270 & $+0.0328$ & 0.310879\\
65 & 182.328845 & 0.757767 & $-0.5845$ & 0.301346\\
256 & 1054.668929 & 0.766876 & $+0.0301$ & 0.309030\\
257 & 1059.768352 & 0.764475 & $-0.5868$ & 0.306628\\
1024 & 5619.499162 & 0.766787 & $+0.0294$ & 0.308576\\
1025 & 5625.980514 & 0.766185 & $-0.5874$ & 0.307974\\
4096 & 28133.365914 & 0.766766 & $+0.0293$ & 0.308463\\
4097 & 28141.232464 & 0.766615 & $-0.5876$ & 0.308312\\
\bottomrule
\end{tabular}

\end{table}

\section{A second proof of the alternating sum}\label{sec:involution}
This appendix derives the alternating sum directly from $W(N)=\per M$ by a sign-reversing involution, without Motzkin paths. Write $M=\JJ+U$, where $\JJ$ is the all-ones matrix. For $N=7$ ($m=6$) the ranks of the rows are $(1,2,3,3,2,1)$ and
\[
  M=\begin{pmatrix}2&1&1&1&1&1\\2&2&1&1&1&0\\2&2&2&1&0&0\\2&2&2&1&0&0\\2&2&1&1&1&0\\2&1&1&1&1&1\end{pmatrix},\qquad
  U=\begin{pmatrix}1&0&0&0&0&0\\1&1&0&0&0&-1\\1&1&1&0&-1&-1\\1&1&1&0&-1&-1\\1&1&0&0&0&-1\\1&0&0&0&0&0\end{pmatrix}.
\]
Column $6$ is the negative of column $2$, column $5$ the negative of column $3$, column $4$ is zero and column $1$ is all ones; the nonzero rows of column $2$ form $C_2=\{2,3,4,5\}$, those of column $3$ form $C_3=\{3,4\}\subset C_2$. In general:

\begin{lemma}\label{lem:colpair}
\begin{enumerate}[label=(\roman*),leftmargin=2.2em,itemsep=1pt]
\item $U_{u,d}=[\rho(u)\ge d]-[\rho(u)\ge N+1-d]$.
\item For $2\le d\le m$, column $N+1-d$ is the negative of column $d$. Column $1$ is all ones; for odd $N$, column $\frac{N+1}2$ is zero.
\item For $2\le d\le R$, $U_{u,d}=[u\in C_d]$ with $C_d:=\{d,\dots,N-d\}$, $|C_d|=N+1-2d$, and $C_R\subset\cdots\subset C_3\subset C_2$.
      So there are $R-1$ column pairs $\{d,N+1-d\}$ ($2\le d\le R$), whose supports have sizes $2,4,\dots,m-2$ ($m$ even) or $1,3,\dots,m-2$ ($m$ odd).
\end{enumerate}
\end{lemma}
\begin{proof}
(i) $\{u,N-u\}=\{\rho,N-\rho\}$ with $\rho=\rho(u)$, so $M_{u,d}=[d\le\rho]+[d\le N-\rho]$, and $[d\le N-\rho]-1=-[\rho\ge N+1-d]$.
(ii) Replacing $d$ by $N+1-d$ in (i) exchanges the two brackets. For $d=1$, $\rho\ge1$ always holds and $\rho\ge N$ never does. If $d=N+1-d$ the column equals its own negative.
(iii) For $d\le R$ we have $N+1-d>N/2\ge\rho$, so the second bracket vanishes, and $\rho(u)\ge d$ if and only if $d\le u\le N-d$.
\end{proof}

\begin{theorem}[Alternating sum, second form]\label{thm:altsum2}
Let $r_j$ be the number of ways to choose $j$ of the $R-1$ column pairs and one row in each of the chosen $2j$ columns, the rows being distinct and both rows chosen for pair $d$ lying in $C_d$ ($r_0=1$). Then
\[
  W(N)=2\sum_{j=0}^{R-1}(-1)^j\,r_j\,(m-2j)!\qquad(N\ge2).
\]
\end{theorem}
\noindent The two forms agree term by term, that is, $r_j=t_{R-1,R-1-j}$, where $t$ is the triangle of Lemma~\ref{lem:trec} for the parity of $N$, so Theorem~\ref{thm:altsum2} is Theorem~\ref{thm:altsum}.
To see this for odd $N$, let the chosen pairs have supports of sizes $2a_1<\dots<2a_j$ ($1\le a_i\le R-1$) and place their rooks in this order. The supports are nested, so when the $i$th pair is placed, its support contains the $2(i-1)$ rows already used, and there are $(2a_i-2i+2)(2a_i-2i+1)=\lambda^{\mathrm o}_{b_i}$ choices, where $b_i:=a_i-i$. As $0\le b_1\le\dots\le b_j\le R-1-j$, summing over the choices gives $r_j=h_j(\lambda^{\mathrm o}_0,\dots,\lambda^{\mathrm o}_{R-1-j})=t^{\mathrm o}_{R-1,R-1-j}$. For even $N$ the sizes are $2a_i-1$ and the number of choices is $\lambda^{\mathrm e}_{b_i}$.
\begin{proof}
\emph{Expansion.} Since $M=\JJ+U$,
\[
  W(N)=\sum_\delta\prod_{u}\bigl(1+U_{u,\delta(u)}\bigr)=\sum_\delta\sum_{I'}\prod_{u\in I'}U_{u,\delta(u)} ,
\]
where $\delta$ runs over the permutations of $\{1,\dots,m\}$ and $I'$ over the sets of rows. The pair $(I',\delta|_{I'})$ is a \emph{placement} $P$ of $|P|:=|I'|$ rooks in distinct rows and columns, with weight $\mathrm{wt}(P):=\prod_{u\in I'}U_{u,\delta(u)}$. A placement extends to a permutation in $(m-|P|)!$ ways, so $W(N)=\sum_{P}(m-|P|)!\,\mathrm{wt}(P)$ over all placements, including the empty one; only placements with all rooks on nonzero entries of $U$ contribute.

\emph{A sign-reversing involution.} A column pair $\{d,N+1-d\}$ is \emph{half used} by $P$ if $P$ uses exactly one of its columns. If $P$ has a half-used pair, take the one with smallest $d$ and move its rook within its row to the other column of the pair, obtaining $P'$. That column was empty, so $P'$ is a placement with the same number of rooks, and $\mathrm{wt}(P')=-\mathrm{wt}(P)$ by Lemma~\ref{lem:colpair}(ii). $P'$ has the same half-used pairs as $P$, so the operation is an involution, and these placements cancel in pairs.

\emph{The remaining terms.} In each remaining placement every column pair is used fully or not at all, and the zero column (for odd $N$) is not used. Suppose $j$ pairs are used. The two rooks of pair $\{d,N+1-d\}$, at $(u,d)$ and $(u',N+1-d)$, contribute $U_{u,d}\,U_{u',N+1-d}=-U_{u,d}\,U_{u',d}=-[u,u'\in C_d]$ by Lemma~\ref{lem:colpair}(iii). So these $2j$ rooks contribute $(-1)^j$, and the number of nonzero choices is $r_j$ by definition.

\emph{Column $1$.} Column $1$ belongs to no pair. If it is not used, $|P|=2j$ and the contribution is $(-1)^jr_j\,(m-2j)!$. If it is used, its rook is in one of the $m-2j\ge1$ remaining rows, $|P|=2j+1$, and the contribution is $(-1)^jr_j\,(m-2j)(m-2j-1)!$, the same. This gives the factor $2$.
\end{proof}

\end{document}